\documentclass[10pt,a4paper]{article}

\usepackage[utf8]{inputenc}
\usepackage[english]{babel}
\usepackage[T1]{fontenc}
\usepackage{amsmath}
\usepackage{amsfonts}
\usepackage{amssymb}
\usepackage{amsthm}
\usepackage{stmaryrd}
\usepackage{tikz}
\usepackage{tikz-cd}
\usepackage{enumitem}
\usepackage{todonotes} 
\usepackage{graphicx}
\usepackage{hyperref}

\DeclareMathOperator{\Hom}{Hom}

\DeclareMathOperator{\Trd}{Trd}

\DeclareMathOperator{\Br}{Br}

\DeclareMathOperator{\Inv}{Inv}

\DeclareMathOperator{\Id}{Id}

\DeclareMathOperator{\Quad}{Quad}
\DeclareMathOperator{\Herm}{Herm}

\DeclareMathOperator{\gr}{gr}

\newcommand{\isom}{\stackrel{\sim}{\rightarrow}}

\newcommand{\Z}{\mathbb{Z}}
\newcommand{\N}{\mathbb{N}}

\newcommand{\pfis}[1]{\langle\!\langle #1\rangle\!\rangle}
\newcommand{\To}{\longrightarrow}

\newcommand{\fdiag}[1]{\langle #1\rangle}

\newcommand{\tld}{\widetilde}
\newcommand{\eps}{\varepsilon}

\newcommand{\pgq}{\geqslant}
\newcommand{\ppq}{\leqslant}

\newcommand{\filt}[1][\N]{\mathbf{Ab}_{#1-\text{filt}}}
\newcommand{\grad}[1][\N]{\mathbf{Ab}_{#1-\text{grad}}}

\renewcommand{\phi}{\varphi}
\renewcommand{\bar}{\overline}
\renewcommand{\hat}{\widehat}

\newcommand{\foncdef}[5]{\begin{array}{rrcl}
#1 : & #2 & \To & #3 \\
 & #4 & \longmapsto & #5
\end{array}}

\newcommand{\anonfoncdef}[4]{\begin{array}{rcl}
#1 & \To & #2 \\
#3 & \longmapsto & #4
\end{array}}

\newtheorem{thm}{Theorem}[section]
\newtheorem{prop}[thm]{Proposition}
\newtheorem{coro}[thm]{Corollary}
\newtheorem{lem}[thm]{Lemma}
\newtheorem{defi}[thm]{Definition}

\theoremstyle{definition}
\newtheorem{rem}[thm]{Remark}
\newtheorem{ex}[thm]{Example}

\author{Nicolas Garrel}
\title{Even Stiefel-Whitney invariants for anti-hermitian quaternionic forms}
\date{}

\begin{document}

\maketitle

\section*{Introduction}

In \cite{GMS}, Serre defines the notion of cohomological invariant of an algebraic group.
More generally, if $K$ is a base field, $\mathbf{Field}_{/K}$ is the category of fields over $K$,
and given functors $F: \mathbf{Field}_{/K}\to \mathbf{Set}$ and $A: \mathbf{Field}_{/K}\to
\mathbf{Ab}$, an invariant of $F$ with values in $A$ is simply a natural transformation
from $F$ to $A$, where $A$ is seen as a set-valued invariant by forgetting the group structure.
In other words, if $\alpha\in \Inv(F,A)$ is such an invariant, then for any extension $L/K$
and any $x\in F(L)$, it yields an element $\alpha(L)\in A(L)$, and this is compatible with
scalar extensions in the sense that if $E/L$ is an over-extension then $\alpha(x_E)=
\alpha(x)_E\in A(E)$. We always assume in this article that the base field has characteristic
different from $2$.

We speak of invariants of an algebraic group $G$ over $K$ when $F(L)$ is the cohomology set
$H^1(L,G)$, which can be identified with the set of isomorphism classes of $G_L$-torsors,
and we speak of cohomological invariants when $A(L)=H^d(L,C)$ where $C$ is some Galois
module defined over $K$ (we then say that the cohomological invariants have coefficients
in $C$). We also consider Witt invariants, which correspond to $A(L)=W(L)$, the Witt group of $L$.
When $G$ is a classical group, the corresponding functor $F$ usually has some natural algebraic
reformulation in terms of bilinear forms or algebras with involution (see \cite{GMS} and \cite{BOI}).

When $G=O(A,\sigma)$ where $(A,\sigma)$ is a central simple algebra with orthogonal involution
over $K$, then $H^1(L,G)$ can be identified with the set of isometry classes of hermitian forms
over $(A,\sigma)$ of reduced dimension $n=\deg(A)$ (\cite[29.26]{BOI}). Up to isomorphism,
this functor only depends on $n$ and the Brauer class $[A]\in \Br(K)$. In particular,
if $G=O(V,q)$ where $(V,q)$ is a non-degenerate quadratic space of dimension $n$, this functor
can be identified with $\Quad_n$, the functor of isometry classes of $n$-dimensional non-degenerate
quadratic forms. In this case, Serre gave in \cite{GMS} a complete description of Witt and
cohomological invariants: the Witt invariants are combinations of $\lambda$-operations $\lambda^d$,
and the cohomological invariants are combinations of Stiefel-Whitney invariants $w_d$.

When $A$ is not split, the cohomological invariants of $O(A,\sigma)$ are only known
in cohomological degree up to $3$ (see \cite{BOI} for degree $1$ and $2$, and \cite{Mer}
for degree $3$). When the index of $A$ is $2$, the problem amounts to finding cohomological
invariants of skew-hermitian forms of reduced dimension $n$ (for some $n$) over some quaternion
algebra $Q$ endowed with its canonical symplectic involution $\gamma$. Some progress was
made in \cite{Ber}, which uses descent methods from the generic splitting field of $Q$ (the function
field of its Sever-Brauer variety, which is somewhat understood as it is the function field
of a conic) to extend Stiefel-Whitney invariants to such hermitian forms. Unfortunately,
this only yields invariants in $H^d(K,\mu_4)/[Q]\cdot H^{d-2}(K,\mu_2)$ instead of $H^d(K,\mu_2)$,
and more importantly the argument is flawed, and actually only works for an invariant
which is well-defined on similarity classes instead of isometry classes (what we call here
an \emph{even} invariant), in particular for what we call here the even Stiefel-Whitney
invariants.

Our main result (Theorem \ref{thm_inv}) is that actually all cohomological invariants of
similarity classes of quadratic forms (otherwise described as \emph{even} invariants) 
do extend to invariants of skew-hermitian quaternionic
forms (not with values in some quotient). Our method is to lift cohomological invariants
to Witt invariants, and use the fact that $\lambda$-operations are also defined for $\eps$-hermitian
forms over algebras with involutions (\cite{G2}) to extend those Witt invariants to such
hermitian forms. We discuss this general method in Section \ref{sec_ext}. 

Since our method
involves lifting invariants with values in a graded ring (mod $2$ Galois cohomology) to a 
filtered ring (the Witt ring), we develop in Section \ref{sec_filt} general tools to
work in this setting.

In Section \ref{sec_split} we study the combinatorics of Witt and cohomological invariants
of families of quadratic forms, as well as the behaviour of even and odd invariants. 
Our point of view is that diagonalizing a hermitian form
amounts to describing it in terms of a family of smaller-dimensional forms. In the case
of quadratic forms, those smaller forms are of dimension 1, and then all basic invariants
are easy to describe combinatorially. When working over a non-split algebra, the smallest
possible dimension for those forms in the decomposition is given by the index of the algebra,
and even for quadratic forms the combinatorics involved to describe the basic invariants
in terms of forms of dimension higher than $1$ become more complicated, and that is what
we explain here. As our main results are for quaternion algebras, the case of a decomposition
as a sum of $2$-dimensional forms is of special interest to us, and we give a complete
description in that case, which paves the way for the extension of those invariants in the next
section.

Finally, in Section \ref{sec_quater}, which is the heart of the article, we prove our main
results on anti-hermitian quaternionic forms. We apply our method of lifting cohomological
invariants to Witt invariants and extending those using our $\lambda$-operations, and we
check that the combinatorial behaviour that was observed in Section \ref{sec_m_2} carries
mostly over for quaternionic forms. What we observe is that a straightforward extension
of the Witt invariant decomposes as a constant part (which becomes trivial over any
splitting extension) and a part which is the direct analogue as the invariant in the
split case. We can then subtract the constant part and take the quotient to cohomology
to get our extended cohomological invariant.

\section{Invariants in a filtered group}\label{sec_filt}

In this section we study a few generalities regarding invariants with values in $\N$-filtered groups.
It would be possible to include more general filtering sets, but at the cost of a few technicalities,
and we only need the case of $\N$ for our application.

\subsection{Filtered and graded functors}

To set up notations, if $A$ is an $\N$-filtered group we write $A_{\pgq n}\subset A$ for
the corresponding subgroup for each $n\in \N$, with
$A_{\pgq n+1}\subset A_{\pgq n}$ as the notation suggests (we also always assume that
$A_{\pgq 0}=A$), and then the induced $\N$-graded group is $\gr(A)=\bigoplus A_n$ with
$A_n=A_{\pgq n}/A_{\pgq n+1}$. Any $\N$-graded group $B=\bigoplus B_n$ is also naturally filtered,
with $B_{\pgq n}=\bigoplus_{m\pgq n} B_m$, and in that case $\gr(B)\simeq B$ canonically.

We define $\filt$ to be the category of $\N$-filtered abelian group, and $\grad$ to be the
category of $\N$-graded abelian groups. Of course a morphism of filtered groups is a group
morphism $f: A\to B$ such that $f(A_{\pgq n})\subset B_{\pgq n}$ for all $n\in \N$.
The previous construction $A\mapsto \gr(A)$ defines a functor $\gr: \filt\to \grad$ in a clear way,
and the canonical filtering of a graded group defines a functor $\grad\to \filt$ such
that the composition is isomorphic to the identity of $\grad$.

Let $F: \mathbf{Field}_{/K}\to \mathbf{Set}$ and $A: \mathbf{Field}_{/K}\to \filt$ be two functors.
There are obvious subfunctors $A_{\pgq n}\subset A$ for each $n\in \N$.
We write $B = \gr\circ A: \mathbf{Field}_{/K}\to \grad$. Since $B(L)$ is by definition
graded for each $L/K$, we get groups $B_n(L)$, which define functors
$B_n: \mathbf{Field}_{/K}\to \mathbf{Ab}$; by definition, for any extension $L/K$,
we simply have $B(L)=\bigoplus_{n\in \N} B_n(L)$, and $B_n(L)=A_{\pgq n}(L)/A_{\pgq n+1}(L)$.
Composing $B$ with the canonical $\grad\to \filt$, we may also see $B$ as a functor
to filtered groups, which just amounts to taking $B_{\pgq n}(L)=\bigoplus_{m\pgq n} B_m(L)$.

\begin{ex}
  If $A(L)=W(L)$ is the Witt group of the field $L$ (with the obvious functor structure),
  endowed with its fundamental filtration $W_{\pgq n}(L)=I^n(L)$, then $B_n(L)=H^n(L,\mu_2)$
  is canonically the mod $2$ Galois cohomology of $L$ (this is essentially a reformulation
  of the Milnor conjecture, see \cite{Kah} for instance for an exposition of this topic).
\end{ex}

\subsection{Invariants}\label{sec_inv_filt}

Then we are interested in the groups of invariants $M = \Inv(F,A)$ and $N=\Inv(F,B)$
of $F$ with values in $A$ and $B$, that is natural transformations $F \to A$ and
$F\to B$, where we see $A$ and $B$ as functors to $\mathbf{Set}$ by forgetting the
filtered/graded group structure. Clearly, $M$ has a canonical structure of abelian
group given by pointwise addition, and even a structure of $\N$-filtered group if we define
$M_{\pgq n}$ to be the image of the natural map $\Inv(F,A_{\pgq n})\to \Inv(F,A)$. In fact,
$M_{\pgq n}$ is nothing but the subgroup of invariants $\alpha\in M$ such that
for all $L/K$ and all $x\in F(L)$, $\alpha_L(x)\in A_{\pgq n}(L)\subset A(L)$.
The same analysis goes for $N$, seeing $B$ as a functor to $\filt$.

Note that even though the filtered structure on $B$ induces a filtered structure
on $N$, this does not work the same way with the graded structure: $N$ is \emph{not}
graded in a natural way. In fact, $\gr(N)$ is $\bar{N} = \bigoplus_{n\in \N} N_n$,
where of course $N_n=\Inv(F,B_n)$, and the natural inclusion $\bar{N}\subset N$
is in general strict. The invariants in $N$ are \emph{locally} of bounded degree:
if $\alpha\in N$ then, for any $x\in F(L)$, there is $m\in \N$ depending
on $x$ such that $\alpha_L(x)$ is a combination of elements of $B_n(L)$ for $n\ppq m$.
Then the invariants in $\bar{N}$ are those with \emph{globally} bounded degree.
Also note that in general $\bar{N}$ is \emph{not} the graded group induced by $M$,
as the canonical map $M_{\pgq n}\to N_n$ has no reason to be surjective
(but its kernel is indeed $M_{\pgq n+1}$, so there is an injective morphism
of graded groups $\gr(M)\to \gr(N)=\bar{N}$). In fact:

\begin{defi}
  We say that an invariant $\alpha\in N_n$ is liftable if it lies in
  the image of $M_{\pgq n}\to N_n$. The subgroup of liftable invariants is denoted
  by $\tld{N}_n$, and we set $\tld{N}=\bigoplus_{n\in \N}\tld{N}_n\subset \bar{N}$.
\end{defi}

Then by definition $\tld{N}_n\simeq M_{\pgq n}/M_{\pgq n+1}$, and $\tld{N}$ is
canonically the graded group induced by $M$. To summarize:
\[ \tld{N} = \gr(M) \subset \bar{N} = \gr(N) \subset N. \]

\begin{ex}
  If $A=W$ is the Witt group functor, and $F=\Quad_r$ (so $F(L)$ is the set of isometry
  classes of non-degenerate quadratic forms of dimension $r$), then all invariants are liftable
  (see Proposition \ref{prop_ic_quad}). It is also the case when $F=I^d$ (see \cite{G3}).
  This means that in those cases $\bar{N}=\tld{N}$.
\end{ex}

\section{Invariants of split hermitian forms}\label{sec_split}

For any functor $F:\mathbf{Field}_{/K}\to \mathbf{Set}$, let us
write $IW(F)$ and $IC(F)$ respectively for the Witt and cohomological invariants
of $F$, that is $IW(F)=\Inv(F,W)$ and $IC(F)=\Inv(F,h^*)$. We also
write $IW^{\pgq d}(F)=\Inv(F,I^d)$ and $IC^d(F)=\Inv(F, h^*)$ (recall
from Section \ref{sec_inv_filt} that this defines a filtering of $IW(F)$, but not exactly
a grading of $IC(F)$).

Let $(A,\sigma)$ be an algebra with involution, and let
$\Herm_n^\eps(A,\sigma): \mathbf{Field}_{/K}\to \mathbf{Set}$ be the functor of
isometry classes of $\eps$-hermitian forms of reduced dimension $n$
over $(A,\sigma)$. Ideally we would like to be able to describe
$IW(\Herm_n^\eps(A,\sigma))$ and $IC(\Herm_n^\eps(A,\sigma))$. In this article
we provide a certain program to tackle this problem, and give significant
advances when $A$ is a quaternion algebra.

But first we take a look at the simplest case: when $A$ is split. We then call
$\eps$-hermitian forms over $(A,\sigma)$ \emph{split hermitian forms}.
When the adjoint involutions of $\eps$-hermitian forms over $(A,\sigma)$ are
symplectic (which happens if $\sigma$ is symplectic and $\eps=1$, or if $\sigma$
is orthogonal and $\eps=-1$), then up to isometry there is only one $\eps$-hermitian
form of reduced dimension $n$ over $(A_L,\sigma_L)$ for all $L/K$, so the
only invariants are constant and the problem is very uninteresting (in contrast,
when $A$ is not split, $\Herm_n^\eps(A,\sigma)$ can be very interesting and complicated).
So for the rest of this section we always assume that either $\sigma$ is orthognal
and $\eps=1$, or $\sigma$ is symplectic and $\eps=-1$.

In that case, choosing a hermitian Morita equivalence between $(A,\sigma)$
and $(K,\Id)$ gives a non-canonical isomorphism of funtors
$\Herm_n^\eps(A,\sigma)\approx \Quad_n$. Choosing a different Morita
equivalence yields the same isomorphism up to multiplying all quadratic
forms by $\fdiag{\lambda}$ for some (uniformly chosen) $\lambda\in K^\times$.
Therefore we get a canonical map
\[ \Herm_n^\eps(A,\sigma) \to \Quad_n/\sim \]
where $\Quad_n/\sim$ is the functor of similarity classes of $n$-dimensional
quadratic forms. This factorizes as
\[ \Herm_n^\eps(A,\sigma) \to \Herm_n^\eps(A,\sigma)/\sim \isom \Quad_n/\sim \]
where again we took similarity classes of hermitian forms.

The non-canonical correspondence with quadratic forms shows that the
invariants of $\Herm_n^\eps(A,\sigma)$ are the same as those of $\Quad_n$, which
are well-understood, but it turns out that because this correspondence is
not canonical, we can't hope to extend all invariants of $\Quad_n$ in the
case where $A$ is not split (there are precise ramification arguments that show
that it's not possible). On the other hand, there is hope for the invariants
of $\Quad_n/\sim$ (and for Witt invariants at least we can show that they indeed
extend).

We know how to describe invariants of $\Quad_n/\sim$, but the combinatorics
involved use the fact that forms in $\Quad_n$ can be diagonalized, in other words
can be decomposed as sums of $1$-dimensional forms. When the index of $A$ is
$m\in \N^*$, we can decompose forms in $\Herm_n^\eps(A,\sigma)$ as sums of
forms of reduced dimension $m$: if $r\in \N^*$ and $X$ is a finite set with
$r$ elements, there is a natural surjective map
\[ \left( \Herm_m^\eps(A,\sigma) \right)^X \to \Herm_{rm}^\eps(A,\sigma) \]
given by $(h_i)_{i\in X}\mapsto \sum_{i\in X}h_i$,
which also induces
\[ \left( \Herm_m^\eps(A,\sigma) \right)^X/\sim \to \Herm_{rm}^\eps(A,\sigma)/\sim \]
where on the left the quotient means that we identify a family $(h_1,\dots,h_r)$
in $\left( \Herm_m^\eps(A,\sigma) \right)^r(L)$ with $(h'_1,\dots,h'_r)$ if there
exists a $\lambda\in L^\times$ independent of $i$ such that $h'_i=\fdiag{\lambda}h_i$.

Thus when $A$ is split we get a commutative diagram:
\[ \begin{tikzcd}
  \Quad_1^X \rar \dar & \Quad_r \dar \\
  \Quad_1^X/\sim \rar & \Quad_r/\sim.
\end{tikzcd} \]

In this section we describe invariants of those functors, and we try to use
a combinatorial setting that does not use diagonalizations of the $m$-dimensional
forms.

\subsection{A coordinate-free approach to Pfister forms}

In this article we promote a slightly different (but strictly equivalent) way to
view and define Pfister forms, which amounts to doing linear algebra without
always choosing a basis.

Traditionally, given elements $a_1,\dots,a_n\in K^\times$, the $n$-fold Pfister form
$\pfis{a_1,\dots,a_n}$ is defined as
\[ \pfis{a_1,\dots,a_n} = \fdiag{1,-a_1}\cdot \fdiag{1,-a_2}\cdots \fdiag{1,-a_n}. \]
The minus signs have no influence on the definition of the \emph{class} of $n$-fold
Pfister forms (and in early texts they were not present), but they are convenient in
the notation when working with quaternion algebras or Galois cohomology.

As a general notation, if $X$ is a set and $(a_i)_{i\in X}$ is a family of elements of
$K^\times$, then for any finite subset $I\subset X$ we write
\[ \fdiag{a_i}_{i\in I} = \sum_{i\in I} \fdiag{a_i} \]
as well as
\[ a_I=\prod_{i\in I}a_i \]
and
\[ \pfis{a_i}_{i\in I} = \pfis{a_{i_1},\dots,a_{i_r}} \]
if $I=\{i_1,\dots,i_r\}$.
We use those notations more generally if the $a_i$ are in $K^\times/(K^\times)^2$.

Let $X$ be a finite set, and $(a_i)_{i\in X}$ be a family of elements
in $K^\times$. By definition,
\[ \pfis{a_i}_{i\in X} = \fdiag{(-1)^{|I|}a_I}_{I\in \mathcal{P}(X)}. \]
Then notice that in $K^\times/(K^\times)^2$:
\[ (-1)^{|I|}a_I \cdot (-1)^{|J|}a_J = (-1)^{|I\Delta J|}a_{I\Delta J} \]
where $I\Delta J$ is the symmetric difference of $I$ and $J$, which
is nothing more than the addition for the standard $\mathbb{F}_2$-vector
space structure on $\mathcal{P}(X)$ (notably, it is the addition when viewed
in terms of characteristic functions $X\to \{0,1\}\simeq \Z/2\Z$).

This motivates the following:

\begin{defi}\label{def_pfis_1}
  Let $V$ be a finite-dimensional $\mathbb{F}_2$-vector space, and
  let $f: V\to K^\times/(K^\times)^2$ be a group morphism (equivalently,
  a linear map). For any affine subspace $W\subset V$, we set
  \[ \pfis{f|W} = \fdiag{f(x)}_{x\in W} \in GW(K). \]
\end{defi}

We choose to use the terminology of $\mathbb{F}_2$-vector spaces rather
than $2$-torsion groups because of the use of affine subspaces, which
do not have a clear equivalent in the theory of abelian groups.

Our analysis preceding the definition shows that if $W$ is a vector subspace
$\pfis{f|W}$ is an $n$-fold Pfister form where $n=\dim(W)$, and
when $W$ is any $n$-dimensional affine subspace $\pfis{f|W}$ is
an $n$-fold general Pfister form (recall that a general Pfister form
is any form that can be written $\fdiag{a}\phi$ with $\phi$ a Pfister form).
Explicitly, suppose $W=w+W_0$ with $W_0\subset $ a vector subspace,
$(e_i)_{i\in X}$ is an $\mathbb{F}_2$-basis of $W_0$, $a\in K^\times$ is
a representative of $f(w)$, and $-a_i$ a representative of $f(e_i)$
for each $i\in X$. Then
\[ \pfis{f|W} = \fdiag{a}\pfis{a_i}_{i\in X}. \]

We actually need a slightly more general case:

\begin{defi}\label{def_pfis_2}
  Let $q$ be a quadratic form over $K$, with similarity factors group $G(q)\subset K^\times$.
  Let $V$ be a finite-dimensional $\mathbb{F}_2$-vector space, and
  let $f: V\to K^\times/G(q)$ be a group morphism. Then for any affine
  subspace $W\subset V$ we define
  \[ \pfis{f|W}q = \sum_{x\in W}\fdiag{f(x)}q \in GW(K). \]
  If $q=\pfis{f'|W'}$ for some $f':V'\to K^\times/(K^\times)^2$ as in
  Definition \ref{def_pfis_1}, we write $\pfis{f|W ; f'|W'}$.
\end{defi}

If $a\equiv b\mod G(q)$ with $a,b\in K^\times$, then by definition of $G(q)$ we have
$\fdiag{a}q=\fdiag{b}q$ in $GW(K)$, so $\fdiag{a}q$ is well-defined for $a\in K^\times/G(q)$,
and thus $\pfis{f|W}q$ is well-defined. Note that when $q=\fdiag{1}$ we recover Definition
\ref{def_pfis_1}.

\begin{prop}
  With the notations of Definition \ref{def_pfis_2}, if $\dim(W)=n$
  and $q$ is a general $r$-fold Pfister form, then $\pfis{f|W}q$ is
  a general $(r+n)$-fold Pfister form.

  If moreover $W$ is a vector subspace and $q$ is a Pfister form,
  then $\pfis{f|W}q$ is a Pfister form.
\end{prop}

\begin{proof}
  Suppose $q=\fdiag{a}\phi$ for some $a\in K^\times$ and $\phi$ an $r$-fold Pfister form,
  and let us write $W=w+W_0$ where $W_0$ is a vector space. Then, recalling that
  $G(q)=G(\phi)$:
  \[ \pfis{f|W}q = \sum_{x\in W_0}\fdiag{f(w)\cdot f(x)}\cdot \fdiag{a}\phi =
    \fdiag{af(w)}\pfis{f|W_0}\phi \]
  so it is enough to prove the second statement, which follows from the same
  analysis as we did after Definition \ref{def_pfis_1}.
\end{proof}

\begin{rem}
  The proof shows that for any $q$, there exists a (possibly general) Pfister form $\phi$
  such that $\pfis{f|W}q=\phi q$, but this $\phi$ is not well-defined and depends
  on a choice of basis and a choice of representatives, while $\pfis{f|W}q$ does not
  depend on any choice.
\end{rem}

\begin{ex}\label{ex_even_part_pfister}
  Let $\phi=\pfis{a_i}_{i\in X}$ be an $n$-fold Pfister form.
  As we saw, we can naturally write it as $\phi=\pfis{f|\mathcal{P}(X)}$
  by setting $f(\{i\})=-a_i$. Then $\mathcal{P}_0(X)$ is a hyperplane
  of $\mathcal{P}(X)$, so we get an $(n-1)$-fold Pfister form
  $\pfis{f|\mathcal{P}_0(X)}$, which we call the even part of $\phi$.
  (Technically this does not depend simply on $\phi$ but also on
  its representation as $\pfis{a_i}_{i\in X}$.)

  If $X=\{1,\dots,n\}$, we can write it as for instance
  $\pfis{a_1a_2,a_1a_3,\dots,a_1a_n}$, or $\pfis{a_1a_2,a_2a_3,\dots,a_{n-1}a_n}$.
  The various ways to represent it as a Pfister form with the usual notation,
  correspond to different choices of basis of $\mathcal{P}_0(X)$, and there is
  no obvious natural choice.
\end{ex}

\subsection{Even and odd invariants}

We now describe a general framework to handle invariants of similarity classes.
It has some overlap with \cite{G3}, but the formalism used there is a little
heavy-handed for our purposes, so we give a self-contained account.

Let $F:\mathbf{Field}_{/K}\to \mathbf{Set}$ be a functor.
The square classes functor $G:L\mapsto L^\times/(L^\times)^2$ is a group functor,
and an $G$-action on $F$ is just an action of $G(L)$ on $F(L)$ for every $L$, which
is compatible in the obvious way with scalar extensions. In that case we can consider
the quotient functor $F/G$ (such that $(F/G)(L)=F(L)/G(L)$ is the usual set-theoretic
quotient for the given action).

\begin{ex}
  Or course our main examples are $F=\Quad_m$, or more generally
  $(\Quad_r)^X$, as well as $\Herm_n^\eps(A,\sigma)$, each time with
  the natural action through multiplication by a $1$-dimensional form.
  The quotients are then the corresponding similarity classes functors.
\end{ex}

For the rest of this section we fix some $F:\mathbf{Field}_{/K}\to \mathbf{Set}$
with a $G$-action. We also write $h^*(L)=H^n(L,\mu_2)$ for any extension $L/K$.

\begin{lem}\label{lem_inv_even}
  The canonical map $IW(F/G)\to IW(F)$ (resp. $IC(F/G)\to IC(F)$) is an injective
  morphism of $W(K)$-algebras (resp. $h^*(K)$-algebras), and its image is the
  subalgebra of invariants $\alpha$ such that for all $L/K$, $x\in F(L)$ and
  $\lambda\in G(L)$, we have $\alpha(\lambda\cdot x)=\alpha(x)$.
\end{lem}

\begin{proof}
  The fact that $F\to F/G$ is surjective by definition implies that $IW(F/G)\to IW(F)$
  and $IC(F/G)\to IC(F)$ are injective. The map on the level of invariants induced
  by some $F\to F'$ is always an algebra morphism, as the algebra structure is defined
  pointwise. Finally, $\alpha$ is in the image of the map if and only if it factorizes
  through an invariant of $F/G$, which is equivalent to the condition in the statement.
\end{proof}

\begin{defi}\label{def_even_odd}
  We say that $\alpha\in IW(F)$ (resp. $\alpha\in IC(F)$) is even if
  it is in the image of $IW(F/G)\to IW(F)$ (resp. $IC(F/G)\to IC(F)$).
  The subalgebra of even invariants is denoted $IW(F)_0$ (resp. $IC(F)_0$).

  We say that $\alpha\in IW(F)$ is odd if for all $L/K$, $x\in F(L)$ and
  $\lambda\in G(L)$, we have $\alpha(\lambda\cdot x)= \fdiag{\lambda}\alpha(x)$.
  We write $IW(F)_1$ for the set of odd invariants in $IW(F)$.
\end{defi}

\begin{ex}
  When $F=\Quad_n$, Serre proved (\cite{GMS}) that $IW(F)$ is a free $W(K)$-module
  with basis $(\lambda^d)_{0\ppq d\ppq n}$. Then the even invariants are those
  that are combinations of the $\lambda^d$ with $d$ even, and the odd invariants
  are those that are combinations of the $\lambda^d$ with $d$ odd.
\end{ex}

Note that there is no notion of odd cohomological invariant.

\begin{prop}\label{prop_even_part_witt}
  $IW(F)_1$ is a $W(K)$-submodule of $IW(F)$, and actually
  \[ IW(F) = IW(F)_0 \oplus IW(F)_1, \]
  which turns $IW(F)$ into a $\Z/2\Z$-graded algebra over $W(K)$.
\end{prop}

\begin{proof}
  The fact that $IW(F)_1$ is a submodule is clear by definition. Let $\alpha\in IW(F)$,
  and let us show that $\alpha=\alpha_0+\alpha_1$ for some uniquely defined
  $\alpha_0$ and $\alpha_1$, respectively even and odd.
  
  Consider some extension $L/K$ and some $x\in F(L)$. Then we can define
  a Witt invariant of $G$ over $L$ by $\beta: \lambda\mapsto \alpha(\lambda\cdot x_E)$
  for all extensions $E/L$ and $\lambda\in G(E)$. According to \cite{GMS}, there
  are uniquely defined elements $u,v\in W(L)$ such that $\beta(\lambda)=u+\fdiag{\lambda}v$
  for all $\lambda\in G(E)$ for all $E/L$. Then we can define $\alpha_0(x)=u$ and $\alpha_1(x)=v$,
  and the uniqueness property of $u$ and $v$ shows that this actually defines invariants
  $\alpha_0,\alpha_1\in IW(F)$, and it is clear by construction that they are
  respectively even and odd. Furthermore, the uniqueness property of $u$ and $v$
  shows that we must have $\alpha_0(x)=u$ and $\alpha_1(x)=v$ for any decomposition
  $\alpha=\alpha_0+\alpha_1$.

  The fact that this gives a $\Z/2Z$-graded algebra just amounts to checking
  that the product of two even/odd invariants is even/odd with the usual rules,
  and that follows directly from the definition.
\end{proof}

\begin{prop}\label{prop_even_part_cohom}
  For any $\alpha\in IC(F)$ there is a unique $\tld{\alpha}\in IC(F)_0$
  such that for all $L/K$, $x\in F(L)$ and $\lambda\in G(L)$, we have
  \[ \alpha(\lambda\cdot x) = \alpha(x) + (\lambda)\cup \tld{\alpha}(x), \]
  and $\alpha$ is even if and only if $\tld{\alpha}=0$.
\end{prop}

\begin{proof}
  Consider some extension $L/K$ and some $x\in F(L)$. Then we can define
  a cohomological invariant of $G$ over $L$ by $\beta: \lambda\mapsto \alpha(\lambda\cdot x_E)$
  for all extensions $E/L$ and $\lambda\in G(E)$. According to \cite{GMS}, there
  are uniquely defined elements $u,v\in h^*(L)$ such that $\beta(\lambda)=u+(\lambda)\cup v$
  for all $\lambda\in G(E)$ for all $E/L$. Taking $\lambda=1$, we get $u=\alpha(x)$.
  Then we can define $\tld{\alpha}(x)=v$, and the uniqueness property of $u$ and $v$ shows
  that this actually defines an invariant $\tld{\alpha}\in IC(F)$, which by definition
  satisfies the expected formula, and is the only possible one by uniqueness of $v$.

  It is clear by definition that $\alpha$ is even if and only if $\tld{\alpha}=0$. It
  just remains to show that $\tld{\alpha}$ is always even. Now consider the following
  formula:
  \begin{align*}
    \alpha((\lambda\mu)\cdot x) &= \alpha(\mu\cdot x) + (\lambda)\cup \tld{\alpha}(\mu\cdot x) \\
                                &= \alpha(x) + (\lambda\mu)\cup \tld{\alpha}(x) \\
                                &= \alpha(x) + (\lambda)\cup \tld{\alpha}(x) +
                                  (\mu)\cup \tld{\alpha}(x).
  \end{align*}
  As this is valid for all $\lambda,\mu\in G(L)$ for all extensions $L$,
  we can take residues at $\lambda$ to get $\tld{\alpha}(\mu\cdot x)=\tld{\alpha}(x)$,
  which shows that $\tld{\alpha}$ is even.
\end{proof}

We now consider degrees of invariants, and the connexion between even Witt
and cohomological invariants:

\begin{prop}\label{prop_deg_even_inv}
  Let $\beta \in IW^{\pgq d}(F)$. Then its even and odd parts $\beta_0$ and $\beta_1$
  are in $IW^{\pgq d-1}(F)$, and have the same class in $IC^{d-1}(F)$. If $\beta$ is
  even or odd, then its class in $IC^d(F)$ is even.

  In fact, if $\alpha\in IC^d(F)$, then $\tld{\alpha}\in IC^{d-1}(F)$, and
  if moreover $\alpha$ is liftable and $\beta\in IW^{\pgq d}(F)$ is a lift of $\alpha$,
  then $\tld{\alpha}$ is also liftable and $\beta_0,\beta_1\in IW^{\pgq d-1}(F)$ are both
  lifts of $\tld{\alpha}$. 
\end{prop}

\begin{proof}
  Let $\beta\in IW^{\pgq d}(F)$.
  It follows from the definition of even/odd invariants that for any
  $L/K$, any $x\in F(L)$ and $\lambda\in G(L)$:
  \begin{equation}\label{eq_tld_pfis}
    \beta(\lambda\cdot x) = \beta(x) - \pfis{\lambda}\beta_1(x). 
  \end{equation}
  In particular, this means that $\pfis{\lambda}\beta_1(x)\in I^d(L)$.
  Since this is true for all $\lambda$ over all extensions of $L$, we can
  take residues at $\lambda$ to get that $\beta_1(x)\in I^{d-1}(L)$ (as
  residue maps satisfy $\partial(I^d)\subset I^{d-1}$, see \cite{GMS}).
  This means that $\beta_1\in IW^{\pgq d-1}$

  Now since $\beta_0+\beta_1=\beta \in IW^{\pgq d}$, we can conclude that
  $\beta_0\equiv -\beta_1$ mod $IW^{\pgq d}(F)$. As $IC^d(F)$ is a $2$-torsion
  group, this means that $\beta_0\equiv \beta_1$ mod $IW^{\pgq d}(F)$, and
  in particular $\beta_0\in IW^{\pgq d-1}(F)$ also.

  If $\beta$ is even or odd, then for all $L/K$, $x\in F(L)$ and $\lambda\in G(L)$:
  $\beta(\lambda\cdot x)\equiv \beta(x)$ mod $I^{d+1}(L)$. This is clear if $\beta$ is
  even (we have equality), and when $\beta$ is odd we use that for any $q\in I^d(L)$,
  $\fdiag{\lambda}q\equiv q$ mod $I^{d+1}(L)$.

  If $\alpha\in IC^d(F)$, then 
  $\alpha(\lambda\cdot x) = \alpha(x) + (\lambda)\cup \tld{\alpha}(x)$ shows that
  $(\lambda)\cup \tld{\alpha}(x)$ is in $h^d(L)$ for all $L/K$, $x\in F(x)$
  and $\lambda\in G(L)$, and taking residues at $\lambda$ shows that $\tld{\alpha(x)}$
  is in $h^{d-1}(L)$.

  Now if we assume that $\alpha$ is the class of $\beta$ in $IC^d(F)$, then
  formula (\ref{eq_tld_pfis}) shows that $-\beta_1$ is a lift of $\tld{\alpha}$
  (and therefore also $\beta_1$). Since $\beta_0$ and $\beta_1$ have the same class
  in $IC^{d-1}(F)$, they are both lifts.
\end{proof}

\begin{rem}\label{rem_lift}
  It would be tempting to guess that any liftable even invariant
  in $IC(F)_0$ has a lift in $IW(F)_0$, but there is actually no
  guarantee that this happens. On the other hand, the even invariants
  that are of the form $\tld{\alpha}$ for some liftable $\alpha$
  do indeed have an even lift, according to the proposition.

  Even assuming that all invariants are liftable, it may very well happen
  that not all even cohomological invariants have the form $\tld{\alpha}$
  (a counter-example is given by $F=I^n$ with $n\pgq 2$, see \cite{G3}).

  On the other hand, if $F=\Quad_{2r}$ for $r\in \N^*$, then all even
  cohomological invariants do have the form $\tld{\alpha}$ (\cite{G3}), and thus
  they all have a lift in $IW(F)_0$.
\end{rem}

\subsection{Invariants of quadratic forms}

Serre gave in \cite{GMS} a complete description of the Witt and cohomological
invariants of $\Quad_r$. Precisely, $IW(\Quad_r)$ is a free $W(K)$-module
with basis $(\lambda^d)_{0\ppq d\ppq r}$, where we recall that $\lambda^d:GW(K)\to GW(K)$
is an operation such that
\[ \lambda^d(\fdiag{a_i}_{i\in X}) = \fdiag{a_I}_{|I|=d}. \]
These operations turn $GW(K)$ into a pre-$\lambda$-ring. 

Actually the $\lambda^d$ are a bit better than just a basis.

\begin{defi}
  Let $F:\mathbf{Field}_{/K}\to \mathbf{Set}$ and $A:\mathbf{Field}_{/K}\to \mathbf{Ring}$
  be functors, and $M=\Inv(F,A)$ the ring of invariants of $F$ with values in $A$.
  For any field extension $L/K$, let $F_L$ and $A_L$ be the restrictions of 
  $F$ and $A$ to $\mathbf{Field}_{/L}$, and $M_L = \Inv(F_L,A_L)$. 

  We say that a family of invariants $(\alpha_i)_i$ in $M$ is a strong basis
  if, for any $L/K$, the family $((\alpha_i)_L)_i$ is an $A(L)$-basis of $M_L$.
\end{defi}

\begin{ex}\label{ex_lambda_basis}
  Since Serre's result is valid over any $L/K$, we get that $(\lambda^d)_{0\ppq d\ppq r}$
  is a strong basis of $IW(\Quad_r)$.
\end{ex}

\begin{prop}\label{prop_strong_basis}
  Let $F:\mathbf{Field}_{/K}\to \mathbf{Set}$ and $A:\mathbf{Field}_{/K}\to \mathbf{Ring}$
  be functors, and $M=\Inv(F,A)$ the ring of invariants of $F$ with values in $A$.
  If $M$ admits a finite strong basis, then any $A(K)$-basis of $M$ is actually a 
  strong basis.
\end{prop}

\begin{proof}
  Let $(\alpha_i)_{1\ppq i\ppq n}$ be a strong basis of $M$, and let 
  $(\beta_i)_{1\ppq i\ppq n}$ be any $A(K)$-basis of $M$. Then the matrix
  $D\in M_n(A(K))$ of the $(\beta_i)_{1\ppq i\ppq n}$ with respect to the basis 
  $(\alpha_i)_{1\ppq i\ppq n}$ is invertible. If $L/K$ is any extension,
  the image $D_L\in M_n(A(L))$ of $D$ by the map $M_n(A(K))\to M_n(A(L))$
  induced by $A(K)\o A(L)$ is the matrix of $((\beta_i)_L)_{1\ppq i\ppq n}$
  with respect to $((\alpha_i)_L)_{1\ppq i\ppq n}$. Since $D$ is invertible,
  so is $D_L$, and therefore $((\beta_i)_L)_{1\ppq i\ppq n}$ is an $A(L)$-basis
  of $M_L$.
\end{proof}

From Example \ref{ex_lambda_basis} and Proposition \ref{prop_strong_basis}
we get:

\begin{coro}\label{cor_strong_basis_w}
  Any $W(K)$-basis of $IW(\Quad_r)$ is a strong basis.
\end{coro}

Let us give another
basis of $IW(\Quad_r)$. In any pre-$\lambda$-ring $R$, one may define
\begin{equation}\label{eq_prd_lambda}
  P_r^d = \sum_{k=0}^d (-1)^k \binom{r-k}{d-k}\lambda^k 
\end{equation}
for all $r\in \N^*$ and $d\in \{0,\dots,r\}$. Clearly, since the family 
$(P_r^d)_{0\ppq d\ppq r}$ is triangular with respect to the basis $(\lambda^d)$
(by which we mean that the transition matrix is triangular with invertible
elements on the diagonal, in this case the diagonal elements being $\pm 1$),
it is also a $W(K)$-basis of $IW(\Quad_r)$, and therefore a strong basis.

Then according to \cite{G3}, a purely formal property of the $P_r^d$, valid
in any pre-$\lambda$-ring $R$, is that
\begin{equation}\label{eq_sum_prd}
  P_{s+t}^d(q_1+q_2) = \sum_{d_1+d_2=d}P_s^{d_1}(q_1)P_t^{d_2}(q_2).
\end{equation}
for any $q_1,q_2\in R$.

\begin{prop}\label{prop_prd}
  Let $X$ be a finite set with $r$ elements, and let $(a_i)_{i\in X}\in (K^\times)^X$
  be a family of elements in $K^\times$. Let us define $f: \mathcal{P}(X)\to K^\times/(K^\times)^2$
  by $f(\{i\})=-a_i$. Then
  \[P_r^d(\fdiag{a_i}_{i\in X}) =   \sum_{|I|=d} \pfis{f|\mathcal{P}(I)}. \]
  In particular, $P_r^d\in IW^{\pgq d}(\Quad_r)$.
\end{prop}

\begin{proof}
  With a simple induction using (\ref{eq_sum_prd}), we see that
  \[ P_r^d(\fdiag{a_i}_{i\in X}) = \sum_{|I|=d} \prod_{i\in I}P_1^1(\fdiag{a_i}), \]
  and we conclude using that $P_1^1(\fdiag{a})=\pfis{a}$, as (\ref{eq_prd_lambda})
  tells us that $P_1^1=\lambda^0-\lambda^1$.

  Since $\pfis{f|\mathcal{P}(I)}$ is a $d$-fold Pfister form, it follows that
  $P_r^d$ has values in $I^d$.
\end{proof}

When $d=r$, we get a single Pfister form:
\[ P_r^r(\fdiag{a_i}_{i\in X}) = \pfis{f|\mathcal{P}(X)} = \pfis{a_i}_{i\in X}. \]

\begin{rem}
  The proposition shows that when $d>r$, $P_r^d$ is zero on $\Quad_r$.
\end{rem}

Now regarding cohomological invariants, \cite{GMS} proves that $IC(\Quad_r)$ is also
a free $h^*(K)$-module, with basis $(w_d)_{0\ppq d\ppq r}$, where the $w_d$ are
the Stiefel-Whitney invariants:
\begin{equation}\label{eq_wd}
  w_d(\fdiag{a_i}_{i\in X}) = \sum_{|I|=d} (a_i)_{i\in I} \in h^d(K) 
\end{equation}
where $(a_i)_{i\in I}$ means the cup-product of all $(a_i)\in h^1(K)$
for $i\in I$.

\begin{rem}
  Since this is valid over any $L/K$, the $w_d$ actually form a strong
  basis of $IC(\Quad_r)$, so any $h^*(K)$-basis of $IC(\Quad_r)$ is a 
  strong basis.
\end{rem}

We record the following obvious but important observation:

\begin{prop}\label{prop_ic_quad}
  The invariant $w_d\in IC(\Quad_r)$ is the class in $IC^d(\Quad_r)$
  of $P_r^d\in IW^{\pgq d}(\Quad_r)$. In particular, all invariants in
  $IC(\Quad_r)$ are liftable.
\end{prop}

\begin{proof}
  Comparing (\ref{eq_wd}) and Proposition \ref{prop_prd}, we may conclude
  using the fact that if $|I|=d$, the class of $\pfis{a_i}_{i\in I}$ in $h^d(K)$
  is exactly $(a_i)_{i\in I}$.
\end{proof}

\subsection{Invariants of families of quadratic forms}

As we explained at the beginning of the section, we are also interested
in invariants of $(\Quad_m)^X$ where $m\in \N^*$ and $X$ is a finite
set with $r$ elements.

For any $I$-indexed family $(\alpha_d)_{d\in I}$ of invariants in $IW(\Quad_m)$,
and any function $\gamma:X\to I$, we may define $\alpha_\gamma\in IW((\Quad_m)^X)$
by
\begin{equation}\label{eq_def_alpha_gamma}
  \alpha_\gamma((q_i)_{i\in X}) = \prod_{i\in X}\alpha_{\gamma(i)}(q_i).
\end{equation}
Likewise, if the $\alpha_d$ are in $IC(\Quad_m)$, we get invariants
$\alpha_\gamma$ in $IC((\Quad_m)^X)$ by taking cup-products.

For any $f: X\to \N$, we write $|f|=\sum_{i\in X}f(i)$.
Then clearly, if $\alpha_d\in IW^{\pgq f(d)}(\Quad_m)$ for some function $f:I\to \N$,
then $\alpha_\gamma\in IW^{\pgq |f\circ \gamma|}((\Quad_m)^X)$. The same thing
holds for degrees of cohomological invariants.

Also, if $\beta_d$ is the class of $\alpha_d$ in $IC^{f(d)}(\Quad_m)$ for all $d\in I$,
then $\beta_\gamma$ is the class of $\alpha_\gamma$ in $IC^{|f\circ \gamma|}((\Quad_m)^X)$.

\begin{lem}\label{lem_basis_gamma}
  Let $(\alpha_d)_{d\in I}$ be a $W(K)$-basis of $IW(\Quad_r)$.
  Then the $\alpha_\gamma$ with $\gamma: X\to I$ form
  a strong basis of $IW((\Quad_m)^X)$. The same result holds for a strong 
  basis of $IC(\Quad_r)$.
\end{lem}

\begin{proof}
  We proceed by induction on $|X|$; if $|X|=1$
  then this is clear. Suppose $X=\{x\}\cup Y$, and let $\alpha\in IW((\Quad_m)^X)$.
  For any $L/K$ and $q\in \Quad_m(L)$, we can define an invariant $\beta_q\in IW((\Quad_m)^Y)$
  over $L$ as follows: if $E/L$ is an extension and $(q_i)_{i\in Y}\in \Quad_m(E)^Y$, then let
  $q_x=q_E\in \Quad_m(E)$, and set $\beta_q((q_i)_{i\in Y})=\alpha((q_i)_{i\in X}$.

  Then $\beta_q$ decomposes uniquely as $\beta_q = \sum_\gamma a_\gamma(q) \alpha_\gamma$
  with $\gamma:Y\to I$ and $a_\gamma(q)\in W(L)$ (we can apply the induction hypothesis
  over the base field $L$ since the $\alpha_d$ remain a basis over $L$ and
  any extension). Then for each $\gamma$, $q\mapsto a_\gamma(q)$ defines an invariant
  in $IW(\Quad_m)$ over $K$, which itself decomposes uniquely as a combination of
  the $\alpha_d$ with coefficients in $W(K)$. Finally, this means that $\alpha$
  itself decomposes uniquely as a combination of the $\alpha_\gamma$ with $\gamma:X\to I$.

  The proof is the same for cohomological invariants.
\end{proof}

We then apply this construction to our basic invariants $P_m^d$ and $w_d$ to get
invariants $P_m^\gamma\in IW^{\pgq |\gamma|}((\Quad_m)^X)$ and
$w_\gamma \in IC^{|\gamma|}((\Quad_m)^X)$ for any $\gamma:X\to \{0,\dots,m\}$.

Explicitly:
\begin{equation}\label{eq_pm_gamma}
  P_m^\gamma((q_i)_{i\in X}) = \prod_{i\in X}P_m^{\gamma(i)}(q_i).
\end{equation}
and
\begin{equation}
  w_\gamma((q_i)_{i\in X}) = \cup_{i\in X}w_{\gamma(i)}(q_i).
\end{equation}

\begin{prop}
  Let $I=\{0,\dots,m\}$. Then the $W(K)$-module $IW((\Quad_m)^X)$
  is free with basis the $P_m^\gamma$ with $\gamma:X\to I$. Likewise,
  $IC((\Quad_m)^X)$ is free with basis the $w_\gamma$, and $w_\gamma$ is
  the class of $P_m^\gamma$ in $IC^{|\gamma|}((\Quad_m)^X)$.
  In particular, all invariants in $IC((\Quad_m)^X)$ are liftable.
\end{prop}

\begin{proof}
  We can apply Lemma \ref{lem_basis_gamma} since the $P_m^d$ form a basis of invariants,
  and that remains valid over any field extension (as we made no assumption
  on the base field). The same applies to $w_\gamma$.
\end{proof}

\begin{prop}\label{prop_decom_prd_gamma}
  We have for any $L/K$ and any $(q_i)_{i\in X}\in (\Quad_m(L))^X$:
  \begin{align*}
    P_{rm}^d\left(\sum_{i\in X}q_i\right) &= \sum_{|\gamma|=d}P_m^\gamma((q_i)_{i\in X}) \\
    w_d\left(\sum_{i\in X}q_i\right) &= \sum_{|\gamma|=d}w_\gamma((q_i)_{i\in X}).
  \end{align*}
  In other words, the image of $P_{rm}^d$ through the canonical map
  $IW(\Quad_{rm})\to IW((\Quad_m)^X)$ is $\sum_{|\gamma|=d}P_m^\gamma$,
  and likewise the image of $w_d$ in $IC((\Quad_m)^X)$ is
  $\sum_{|\gamma|=d}w_\gamma((q_i)_{i\in X})$.
\end{prop}

\begin{proof}
  The computation of $P_m^d(\sum_{i\in X}q_i)$ is by induction on
  the formula (\ref{eq_sum_prd}), and that of $w_d(\sum_{i\in X}q_i)$
  follows.
\end{proof}

\subsection{Invariants of similarity classes}

We now turn to invariants of $\Quad_n/\sim$ and $(\Quad_m)^X/\sim$.
According to Lemma \ref{lem_inv_even}, they are identified with
the subalgebras of even invariants of $\Quad_r$ and $(\Quad_m)^X$
respectively.

We also know that for Witt invariants, the even parts of a system
of generators of $IW(F)$ form a system of generators of $IW(F)_0$
for any functor $F$. This leads to:

\begin{defi}
  We define $Q_n^d\in IW(\Quad/\sim)$ as the even part of $P_n^d\in IW(\Quad_n)$,
  and $Q_m^\gamma\in IW((\Quad_m)^X/\sim)$ as the even part of $P_m^\gamma\in IW((\Quad_m)^X)$,
  for any $\gamma:X\to \{0,\dots,m\}$.
\end{defi}

\begin{rem}
  Be aware that $Q_m^\gamma$ is \emph{not} obtained from the $Q_m^d$ through the process
  described in (\ref{eq_def_alpha_gamma}).
\end{rem}

\begin{lem}\label{lem_basis_gamma_even}
  With the same hypotheses as in Lemma \ref{lem_basis_gamma}, let us
  furthermore assume that there is a function $f:I\to \Z/2\Z$ such that
  $\alpha_d\in IW(\Quad_m)_{f(d)}$ for all $d\in I$ (so the $\alpha_d$
  form a graded basis for the $\Z/2\Z$-grading on $IW(\Quad_m)$).

  For each $\gamma:X\to I$ we write $\pi(\gamma)=\sum_{i\in X}f(\gamma(i))\in \Z/2\Z$.
  Then the $\alpha_\gamma$ are a graded basis of $IW((\Quad_m)^X)$ in the sense
  that $\alpha_\gamma\in IW((\Quad_m)^X)_{\pi(\gamma)}$. In particular, the
  $\alpha_\gamma$ with $\pi(\gamma)=0$ are a basis of $IW((\Quad_m)^X)_0$.   
\end{lem}

\begin{proof}
  We already know from Lemma \ref{lem_basis_gamma} that the $\alpha_\gamma$
  form  a basis, and the parity of $\alpha_\gamma$ is immediate from
  the fact that $IW((\Quad_m)^X)$ is a $Z/2\Z$-graded algebra, and
  the definition of $\alpha_\gamma$ as a product.
\end{proof}

\begin{prop}\label{prop_iw_sim}
  We have $Q_n^d\in IW^{\pgq d-1}(\Quad_n/\sim)$, and $IW(\Quad_n/\sim)$
  is free with strong basis the $Q_n^d$ with $d\in \{0,\dots,n\}$ which is even.
  Moreover, all invariants in $IC(\Quad_n/\sim)$ are liftable.

  Let $I=\{0,\dots,m\}$. For any $\gamma: X\to I$, $Q_m^\gamma\in
  IW^{\pgq |\gamma|-1}((\Quad_m)^X/\sim)$,
  and $IW((\Quad_m)^X/\sim)$ is free with strong basis the $Q_m^\gamma$ with
  $\gamma:X\to I$ such that $|\gamma|$ is even.
\end{prop}

\begin{proof}
  The fact that $Q_n^d$ is in $IW^{\pgq d-1}$ is a direct consequence of
  Proposition \ref{prop_deg_even_inv}, as $P_n^d\in IW^{\pgq d}$. Since
  the $\lambda^d$ are even/odd according to the parity of $d$, the $\lambda^d$
  with $d$ even form a basis of $IW(\Quad_n/\sim)$, and actually a strong
  basis since this holds over any $L/K$. Now $P_n^d$ is the sum of
  $\lambda^d$ plus combinations of $\lambda^k$ with $k<d$, so when $d$ is
  even $\lambda^d$ is also the leading term of $Q_n^d$, which means that the
  family $(Q_n^d)_{d \text{ even}}$ is triangular in the basis $(\lambda^d)_{d \text{ even}}$,
  so it is also a basis.

  The fact that invariants in $IC(\Quad_n/\sim)$ are liftable follows from
  the discussion in Remark \ref{rem_lift}: all invariants in $IC(\Quad_n)$ are
  liftable, so it is enough to show that all even invariants in $IC(\Quad_n)$ are
  of the form $\tld{\alpha}$ for some $\alpha$. This follows from explicit
  basis computations in \cite{G3}: there is a basis $v_d$ of $IC(\Quad_n)$ such
  that $IC(\Quad_n/\sim)$ is the submodule generated by the $v_d$ with $d$ odd
  \cite[Rem 9.13]{G3}, and we have $\tld{v_{d+1}}=v_d$ when $d$ is odd
  \cite[Prop 7.6]{G3}.
  
  To show the statement regarding the $Q_n^\gamma$, we cannot directly apply Lemma
  \ref{lem_basis_gamma_even} to $\alpha_d=P_n^d$, as they are neither even nor odd,
  but we rather apply it to $\alpha_d=\lambda^d$. Then again, for the same reason as
  for the $Q_n^d$, the $Q_n^\gamma$ with $|\gamma|$ even form a triangular family in
  the basis $(\alpha_\gamma)_{|\gamma| \text{ even}}$. As this holds over any $L/K$
  it is a strong basis.
\end{proof}

\begin{prop}\label{prop_decom_qrd_gamma}
  We have for any $L/K$ and any $(q_i)_{i\in X}\in (\Quad_m(L))^X$:
  \[ Q_{rm}^d(\sum_{i\in X}q_i) = \sum_{|\gamma|=d}Q_m^\gamma((q_i)_{i\in X}). \]
  In other words, the image of $Q_{rm}^d$ through the canonical map
  $IW(\Quad_{rm}/\sim)\to IW((\Quad_m)^X/\sim)$ is $\sum_{|\gamma|=d}Q_m^\gamma$.
\end{prop}

\begin{proof}
  This is just taking the even part on both sides of Proposition
  \ref{prop_decom_prd_gamma}.
\end{proof}

\subsection{The splitting process}

We now have basic Witt invariants $P_{rm}^d$, $Q_{rm}^d$, $P_m^\gamma$ and $Q_m^\gamma$
of our respective functors $\Quad_{rm}$, $\Quad_{rm}/\sim$, $(\Quad_m)^X$ and $(\Quad_m)^X/\sim$.
We know that they have values in $I^k$ for certain $k\in \N$ (respectively,
$k=d$, $k=d-1$, $k=|\gamma|$ and $k=|\gamma|-1$), which means that their
values can be written as sums of $k$-fold Pfister forms in each case, and we would
like to understand what these Pfister forms look like given the combinatorics
of our quadratic forms.

As we explained, the idea of studying invariants of $(\Quad_m)^X$ and $(\Quad_m)^X/\sim$
is that they are stepping stones to the case of hermitian forms over algebras whose
index divides $m$. Now suppose we are studying hermitian forms over such algebras,
but we want to understand how to specialize to algebras that actually have a smaller
index, say dividing $m'$ with $m'=rm$. Then each $m$-dimensional form can be
further decomposed as a sum of $m'$-dimensional forms indexed by some set $Y$ with
$r$ elements. The case where $m'=1$ corresponds to working with split algebras, and
fully diagonalizing each quadratic form. In general, the situation can be modeled
with the canonical morphism
\[ (\Quad_m)^{X\times Y}\to (\Quad_{m|Y|})^X \]
for some finite sets $X$ and $Y$, which is defined by
\[ (q_{i,j})_{i\in X,j\in Y}\mapsto (\sum_{j\in Y}q_{i,j})_{i\in X}. \]

This induces an injection
\[ IW((\Quad_{m|Y|})^X)\to IW((\Quad_m)^{X\times Y}), \]
and one can ask how the basic invariant decompose.
For any $\omega: X\times Y\to \N$, we define $\omega_X:X\to \N$ as
\[ \omega_X(i)= \sum_{j\in Y}\omega(i,j). \]
It satisfies $|\omega_X|=|\omega|$.

\begin{prop}\label{prop_prd_gamma_omega}
  Let $\gamma: X\to \N$. The image of $P_{m|Y|}^\gamma$ in
  $IW((\Quad_m)^{X\times Y})$ is
  \[ \sum_{\omega: X\times Y\to \N, \omega_X=\gamma} P_m^\omega. \]
\end{prop}

\begin{proof}
  Let $q_{i,j}\in \Quad_m(L)$ for every $(i,j)\in X\times Y$, for some $L/K$.
  Then
  \begin{align*}
    P_{m|Y|}^\gamma((\sum_{j\in Y}q_{i,j})_{i\in X})
    &= \prod_{i\in X}P_m^{\gamma(i)}(\sum_{j\in Y}q_{i,j}) \\
    &= \prod_{i\in X}\sum_{f: Y\to \N, |f|=\gamma(i)}P_m^f((q_{i,j})_{j\in Y}) \\
    &= \sum_{F: X\to (Y\to \N), |F(i,-)|=\gamma(i)}\prod_{i\in X}\prod_{i\in X}P_m^{F(i)}((q_{i,j})_{j\in Y}) \\
    &= \sum_{\omega: X\times Y\to \N, \omega_X=\gamma}P_n^\omega((q_{i,j})_{(i,j)\in X\times Y})
  \end{align*}
  where we used Proposition \ref{prop_decom_prd_gamma} and the correspondence
  between $\Hom(X,\Hom(Y,\N))$ and $\Hom(X\times Y, \N)$.
\end{proof}

Note that only the $\gamma$ with image in $\{0,\dots,m|Y|\}$ yield
a non-zero $P_{m|Y|}^\gamma$, and likewise only the $\omega$ with image in
$\{0,\dots,m\}$ give non-zero contribution in the formula.

\begin{coro}
  Let $\gamma: X\to \N$. The image of $Q_{m|Y|}^\gamma$ in
  $IW((\Quad_m)^{X\times Y}/\sim)$ is
  \[ \sum_{\omega:X\times Y\to \N, \omega_X=\gamma} Q_m^\omega. \]
\end{coro}

\begin{proof}
  This is just taking the even part on both sides of Proposition \ref{prop_prd_gamma_omega}.
\end{proof}

\subsection{The case $m=1$}

When we take $m=1$, which means that we fully diagonalize the forms,
there is a nice description of the Pfister forms that intervene.
Note that in that case, our functions $\gamma:X\to \N$ that are
involved in invariants of $(\Quad_1)^X$ only take values in $\{0,1\}$,
so we can identify them with subsets $I\subset X$, using the characteristic
function $\chi_I$ of $I$. Note that $|\chi_I|=|I|$.

\begin{lem}\label{lem_p1_chi}
  Let $X$ be a finite set, and let $I\subset X$. Let  $(a_i)_{i\in X}\in (K^\times)^X$.
  We define a morphism $f: \mathcal{P}(X)\to K^\times/(K^\times)^2$ by $f(\{i\})=-a_i$.
  Then
  \begin{align*}
    P_1^{\chi_I}((\fdiag{a_i})_{i\in X}) &= \pfis{f|\mathcal{P}(I)} \\
    Q_1^{\chi_I}((\fdiag{a_i})_{i\in X}) &= \pfis{f|\mathcal{P}_0(I)}. \\
  \end{align*}
\end{lem}

\begin{proof}
  The first equality simply comes from the fact that $P_1^0(\fdiag{a}))=1$ and
  $P_1^1(\fdiag{a})=\pfis{a}$, so $P_1^{\chi_I}((\fdiag{a_i})_{i\in X})=\prod_{i\in I}\pfis{a_i}$.

  For the second one, notice that if $\lambda\cdot f$ denotes the morphism
  $\mathcal{P}(X)\to K^\times/(K^\times)^2$ defined by $(\lambda\cdot f)(\{i\})=
  \lambda f(\{i\})$, then $(\lambda\cdot f)(I)=f(I)$ if $|I|$ is even,
  and $(\lambda\cdot f)(I)=\lambda f(I)$ if $|I|$ is odd. So
  \[ P_1^{\chi_I}((\fdiag{\lambda a_i})_{i\in X})
    = \pfis{f|\mathcal{P}_0(I)}+\fdiag{\lambda}\pfis{f|\mathcal{P}_1(I)}. \]
  Then by definition of the even part of a Witt invariant we get the formula for $Q_1^{\chi_I}$.
\end{proof}

Now if we want to apply that to $(\Quad_1)^{X\times Y}\to (\Quad_{|Y|})^X$,
we notice that when $\omega:X\times Y\to \N$ is $\chi_I$ for some
$I\subset X\times Y$, then $\omega_X:X\to \N$ is the function $\theta_I$ defined by
\[ \theta_I(i) = |\pi^{-1}(\{i\})\cap I| \]
where $\pi:X\times Y\to X$ is the canonical projection. In particular,
$|\theta_I|=|I|$.

\begin{prop}\label{prop_pq_split}
  Let $(q_i)_{i\in X}\in (\Quad_{|Y|}(K))^X$, and let $q=\sum_{i\in X}q_i$. For each $i\in X$,
  we diagonalize $q_i$ as $\fdiag{a_{i,j}}_{j\in Y}$. Let us define
  $f: \mathcal{P}(X\times Y)\to K^\times/(K^\times)^2$ by $f(\{(i,j)\}) = -a_{i,j}$. Then:
  \begin{align*}
    P_{|X||Y|}^d(q) &= \sum_{|I|=d}\pfis{f|\mathcal{P}(I)} \\
    Q_{|X||Y|}^d(q) &= \sum_{|I|=d}\pfis{f|\mathcal{P}_0(I)} \\
    P_{|Y|}^\gamma((q_i)_{i\in X}) &= \sum_{\theta_I=\gamma}\pfis{f|\mathcal{P}(I)} \\
    Q_{|Y|}^\gamma((q_i)_{i\in X}) &= \sum_{\theta_I=\gamma}\pfis{f|\mathcal{P}_0(I)}.
  \end{align*}
\end{prop}

\begin{proof}
  For the last two equations, we use Proposition \ref{prop_prd_gamma_omega} with $m=1$,
  and the fact that each function $\omega:X\times Y\to \{0,1\}$ can be written as $\chi_I$
  with $I\subset I\times Y$, with $\omega_X=\gamma$ equivalent to $\theta_I=\gamma$.
  We conclude with each $P_1^{\chi_I}$ (resp. $Q_1^{\chi_I}$) replaced by the formula given
  by Lemma \ref{lem_p1_chi}.

  The first two equations are special cases, using $X'=\{*\}$ and $Y'=X\times Y$,
  and identifying a function $\gamma: X'\to \N$ with its value $d$.
\end{proof}

\subsection{The case $m=2$}\label{sec_m_2}

In general, we would like to get a good combinatorial description of
the Pfister forms involved in $P_m^\gamma$ and $Q_m^\gamma$, without
diagonalizing the $m$-dimensional forms (unlike in Proposition \ref{prop_pq_split}),
but we do not yet have one to offer. On the other hand, we give a satisfying
answer when $m=2$, which might not generalize to higher $m$.

Let $q\in \Quad_2(K)$. Instead of diagonalizing it, let us write it as
\[ q = \fdiag{t}\pfis{\delta} \]
where $\delta=\det(q)\in K^\times/(K^\times)^2$ is well-defined, and
$t\in K^\times$ is only well-defined modulo $G(\pfis{\delta})$. Of course
if $q=\fdiag{a,b}$ then $\delta=ab$ and we can take $t=a$ or $t=b$.

Then if $(q_i)_{i\in X}\in (\Quad_2(K))^X$ for some finite $X$, we can
define $\delta_i\in K^\times/(K^\times)^2$ and $t_i\in K^\times/G(\pfis{\delta_i})$
for all $i\in X$. This extends naturally to group morphisms
\[ \foncdef{\delta}{\mathcal{P}(X)}{K^\times/(K^\times)^2}{\{i\}}{-\delta_i} \]
and for any $Y\subset X$:
\[ \foncdef{t}{\mathcal{P}(Y)}{K^\times/G(\pfis{\delta|\mathcal{P}(Y)})}{\{i\}}{-t_i} \]
(note that since $t_i$ is well-defined modulo $G(\pfis{\delta_i})$, for any $I\subset Y$
we get that $t(I)$ is well-defined modulo $G(\pfis{\delta|\mathcal{P}(I)})$, and in particular
modulo $G(\pfis{\delta|\mathcal{P}(Y)})$).

We see, given the definition, that for any $I\subset X$:
\begin{equation}\label{eq_prod_q}
  \prod_{i\in I}q_i = \fdiag{(-1)^{|I|}t(I)}\pfis{\delta|\mathcal{P}(I)}.
\end{equation}

If $V$ is an affine subspace of some vector space, we write $\overrightarrow{V}$
for its direction (which is the unique vector subspace parallel to $V$).

\begin{defi}\label{def_psi_quad}
  Let $U,V\subset \mathcal{P}(X)$ be affine subspaces with $\overrightarrow{V}\subset
  \overrightarrow{U}$. Then we define $\Psi^{V,U}, \Psi_0^{V,U}\in IW((\Quad_2)^X)$ by
  \[ \Psi^{V,U}((q_i)_{i\in X}) = \pfis{t|V; \delta|U} \]
  and
  \[ \Psi_0^{V,U} = \Psi^{V\cap \mathcal{P}_0(X),U} \]

  Furthermore, if $A,J\subset X$ are disjoint, we define 
  $V_{A,J}=J+\mathcal{P}(A)$ and
  \[ \Psi^{J,A} = \Psi^{V_{J,A}, \mathcal{P}(A\cup J)},
    \quad \Psi_0^{J,A} = \Psi_0^{V_{J,A}, \mathcal{P}(A\cup J)}. \]
\end{defi}

The fact that those formulas indeed define invariants is clear since the morphisms
$\delta$ and $t$ are canonically defined from the $q_i$. The crucial observation
is that those invariants take values in general Pfister forms.

\begin{rem}
  It might happen that $V\cap \mathcal{P}_0(X)=\emptyset$, in which case $\Psi_0^{V,U}=0$.
  Otherwise, the direction of $V\cap \mathcal{P}_0(X)$ is $\overrightarrow{V}\cap \mathcal{P}_0(X)$.

  In particular, except when $A=\emptyset$ and $|J|$ is odd, the direction of
  $V_{J,A}\cap \mathcal{P}_0(X)$ is $\mathcal{P}_0(A)$. When $|J|$ is even,
  this affine space is actually $J+\mathcal{P}_0(A)$, but when $|J|$ is odd there is no
  natural basepoint.
\end{rem}

\begin{lem}
  Let $U,V\subset \mathcal{P}(X)$ be affine subspaces with $\overrightarrow{V}\subset
  \overrightarrow{U}$. Then $\Psi_0^{V,U}$ is the even part of $\Psi^{V,U}$.

  In particular, $\Psi_0^{V,U}\in IW^{\dim(U)+\dim(V)-1}((\Quad_2)^X/\sim)$.
\end{lem}

\begin{proof}
  We see from the definition that if $\delta'$ and $t'$ are the morphisms corresponding
  to the family $(q_i')_{i\in X}$ where $q_i'=\fdiag{\lambda}q_i$, then $\delta'=\delta$,
  while $t'(I)=t(I)$ if $|I|$ is even and $t'(I)=\fdiag{\lambda}t(I)$ if $|I|$ is odd.

  This implies that
  \begin{align*}
    \Psi^{V,U}((\fdiag{\lambda}q_i)_{i\in X}) &= \pfis{t'|V; \delta'|U} \\
                     &= \pfis{t|V\cap \mathcal{P}_0(X); \delta|U}
                      +\fdiag{\lambda}\pfis{t|V\cap \mathcal{P}_1(X); \delta|U}
  \end{align*}
  which shows that $\Psi_0^{U,V}$ is the even part.

  The part about degrees follows simply from Proposition \ref{prop_deg_even_inv}.
\end{proof}

We now see that those invariants with values in general Pfister forms actually
generate all invariants.

\begin{prop}\label{prop_pq_m2}
  Let $\gamma:X\to \{0,1,2\}$, and let $A=\gamma^{-1}(\{2\})\subset X$
  and $B=\gamma^{-1}(\{1\})$. Then
  \begin{align*}
    P_2^\gamma &= \sum_{J\subset B}2^{|B\setminus J|}\Psi^{J,A} \\
    Q_2^\gamma &= \sum_{J\subset B}2^{|B\setminus J|}\Psi_0^{J,A}.                             
  \end{align*}
\end{prop}

\begin{proof}
  The second formula follows from the first be taking the even part on both
  sides, so we only prove the first one.

  Let $(q_i)_{i\in X}\in (\Quad_2(K))^X$.
  First observe that from the definition (\ref{eq_prd_lambda}) we get
  \begin{align*}
    P_2^0(q_i) &= 1 \\
    P_2^1(q_i) &= 2-q_i \\
    P_2^2(q_i) &= \pfis{\delta_i} - q_i.
  \end{align*}
  Therefore:
  \begin{align*}
    & P^\gamma((q_i)_{i\in X}) \\
    =& \prod_{i\in A}(\pfis{\delta_i}-q_i)\times \prod_{i\in B}(2-q_i) \\
    =& \sum_{I\subset A, J\subset B}(-1)^{|I|}\fdiag{(-1)^{|I|}t(I)}\pfis{\delta|\mathcal{P}(I)}
       \cdot \pfis{\delta|\mathcal{P}(A\setminus I)}
       \times (-1)^{|J|}\fdiag{(-1)^{|J|}t(J)}\pfis{\delta|\mathcal{P}(J)}\cdot 2^{|B\setminus J|} \\
    =& \sum_{I\subset A, J\subset B} 2^{|B\setminus J|}\fdiag{t(I\cup J)}\pfis{\delta|\mathcal{P}(A\cup J)} \\
    =&\sum_{J\subset B}2^{|B\setminus J|}\pfis{t|J+\mathcal{P}(A); \delta|\mathcal{P}(A\cup J)}. \qedhere
  \end{align*}
\end{proof}

Let us now compare the descriptions of $P_2^\gamma$ given in Propositions
\ref{prop_pq_split} and \ref{prop_pq_m2}. So we now allow ourselves to
look into each $q_i$: this gives us $f:\mathcal{P}(X\times Y)\to K^\times/(K^\times)^2$
with some $Y$ with $2$ elements, and $f$ encodes a diagonalization of each $q_i$.
Recall that $\pi:X\times Y\to X$ is the projection.

We choose a section $s: X\to X\times Y$ of $\pi$. It induces a morphism
$s_*: \mathcal{P}(X)\to \mathcal{P}(X\times Y)$. We also define
\[ \foncdef{\Delta}{\mathcal{P}(X)}{\mathcal{P}(X\times Y)}{I}{I\times Y} \]
which is $\mathbb{F}_2$-linear.

Note that
\begin{equation}\label{eq_s_delta_gen}
  \mathcal{P}(X\times Y)=s_*(\mathcal{P}(X))\oplus \Delta(\mathcal{P}(X))
\end{equation}
and
\[ \mathcal{P}_0(X\times Y)=s_*(\mathcal{P}_0(X))\oplus \Delta(\mathcal{P}(X)). \]

\begin{lem}
  For all $U,V\subset \mathcal{P}(X)$ affine subspaces with $\overrightarrow{V}\subset
  \overrightarrow{U}$, we have
  \[ \pfis{t|V; \delta|U} = \pfis{f|s_*(V)+ \Delta(U)}. \]
\end{lem}

\begin{proof}
  The section $s$ gives us a way to choose representatives in $K^\times/(K^\times)^2$
  for each $t(\{i\})$: we take $f(\{s(i)\})$. In practice, if $q_i=\fdiag{a_i,b_i}$
  for each $I$ is the diagonalization which defines $f$, then $s$ picks between $a_i$ or
  $b_i$ for each $i$, and those are possible representatives for $t_i$.

  Furthermore, $\delta(\{i\})=-a_ib_i=f(\{i\}\times Y)=f(\Delta(\{i\}))$.
  This means that $\delta=f\circ \Delta$ and $t=f\circ s_*$ (at least after
  taking the classes modulo $G(\pfis{\delta|U})$), which gives the formula.
\end{proof}

Therefore when we compare Propositions \ref{prop_pq_split} and \ref{prop_pq_m2},
we should have, given $\gamma: X\to \{0,1,2\}$, $A=\gamma^{-1}(\{2\})\subset X$
and $B=\gamma^{-1}(\{1\})$:
\begin{align}\label{eq_comp_pfis}
  \sum_{I\subset X\times Y, \theta_I=\gamma} \pfis{f|\mathcal{P}(I)}
  &= 2^{|B\setminus J|}\sum_{J\subset B}\pfis{f|s_*(V_{J,A})+\Delta(\mathcal{P}(A\cup J))} \\
  \sum_{I\subset X\times Y, \theta_I=\gamma} \pfis{f|\mathcal{P}_0(I)}
  &= 2^{|B\setminus J|}\sum_{J\subset B}\pfis{f|s_*(V_{J,A}\cap \mathcal{P}_0(X))+\Delta(\mathcal{P}(A\cup J))}.
\end{align}

We only prove the first formula, the second one being similar and just necessiting taking
extra care about which subsets have even cardinality.

Actually, at this point the result is pure combinatorics and has nothing to do
with quadratic forms anymore. Let us consider the group ring $\Z[\mathcal{P}(X\times Y)]$.
For any subset $U\subset \mathcal{P}(X\times Y)$, we set $\sigma(U)=\sum_{x\in U}x\in
\Z[\mathcal{P}(X\times Y)]$. Then formula (\ref{eq_comp_pfis}) boils down to:

\begin{prop}
  Let $\gamma: X\to \{0,1,2\}$, $A=\gamma^{-1}(\{2\})\subset X$
  and $B=\gamma^{-1}(\{1\})$. Then in $\Z[\mathcal{P}(X\times Y)]$,
  we have
  \[ \sum_{I\subset X\times Y, \theta_I=\gamma} \sigma(\mathcal{P}(I))
  = 2^{|B\setminus J|}\sum_{J\subset B}\sigma(s_*(V_{J,A})+\Delta(\mathcal{P}(A\cup J))). \]
\end{prop}

\begin{proof}
  Given the definition of $A$, the subsets $I\subset X\times Y$ such that $\theta_I=\gamma$
  are exactly those that can be written $I=(A\times Y)\coprod I'$ with $I'\subset B\times Y$
  such that $\pi$ induces a bijection $I'\to B$. We then have
  $\mathcal{P}(I)=\mathcal{P}(A\times Y)\oplus \mathcal{P}(I')$.

  Now given $J\subset B$, we have
  \begin{align*}
    s_*(V_J)+\Delta(\mathcal{P}(A\cup J))
    &= s(J)+s_*(\mathcal{P}(A))+\Delta(\mathcal{P}(A))+\Delta(\mathcal{P}(J)) \\
    &= s(J) + \Delta(\mathcal{P}(J)) + \mathcal{P}(A\times Y)
  \end{align*}
  where we used (\ref{eq_s_delta_gen}).

  So we are reduced to
  \begin{equation}\label{eq_combin_fin}
  \sum_I \sigma(\mathcal{P}(I))
    = \sum_{J\subset B}2^{|B\setminus J|}s(J)\sigma(\Delta(\mathcal{P}(J))) 
  \end{equation}
  where the left sum is over subsets $I\subset B\times Y$ such that $\pi$ induces
  a bijection $I\to B$.

  If we fully develop the left sum $\sum_I \sigma(\mathcal{P}(I))$, the elements
  that appear are the $S\subset B\times Y$ such that $\pi$ induces an injection
  $S\to B$, and each such $S$ appears as many times as there are ways to complete $S$
  in as $I$ such that $\pi$ induces a bijection $I\to B$. This means that for each
  element $i\in B\setminus \pi(S)$ there $2$ possible choices for the antecedent of
  $i$ in $I$, so in total $S$ appears $2^{B\setminus \pi(S)}$ times.

  And actually one can see that $s(J)\sigma(\Delta(\mathcal{P}(J)))$ is exactly
  the sum of all subsets $S\subset B\times Y$ such that $\pi$ induces a bijection
  $S\to J$. Indeed, any such $S$ will correspond in the sum given by $\sigma$ to
  the subset $T\subset J$ of elements $i\in J$ such that the antecedent of $i$
  in $S$ is not $s(i)$.

  This establishes (\ref{eq_combin_fin}).
\end{proof}

\section{Invariants of quaternionic forms}\label{sec_quater}

\subsection{Extending invariants to hermitian forms}\label{sec_ext}

Let $(A,\sigma)$ be a central simple algebra with involution of the first kind over $K$,
and let $\eps$ be $1$ if $\sigma$ is orthogonal and $-1$ if $\sigma$ is symplectic.

If $(B,\tau)$ is another such algebra with involution, with corresponding $\eps'\in \{\pm 1\}$,
and if $[A]=[B]\in \Br(K)$, then we can make a choice of hermitian Morita
equivalence between $(A_,\sigma)$ and $(B,\tau)$, which induces an isomorphism
$\Herm_n^\eps(A,\sigma)\to \Herm_n^{\eps'}(B,\tau)$. Another choice of equivalence changes this map
by a multiplicative scalar constant, and therefore there is a \emph{canonical} morphism
$\Herm_n^\eps(A,\sigma)\to \Herm_n^{\eps'}(B,\tau)/\sim$ which actually gives a canonical
isomorphism $\Herm_n^\eps(A,\sigma)(L)/\sim\isom \Herm_n^{\eps'}(B,\tau)/\sim$.

In particular, for any splitting extension $L/K$ of $A$, there is a canonical
isomorphism $\Herm_n^\eps(A_L,\sigma_L)/\sim\to (\Quad_n)_L/\sim$.

\begin{defi}
  Let $\alpha\in IW(\Quad_n/\sim)$. Then we say that an invariant $\hat{\alpha}\in
  IW(\Herm_n^\eps(A,\sigma)/\sim)$ is an extension of $\alpha$ if for
  any splitting extension $L/K$ of $A$, $\hat{\alpha}_L$ corresponds
  to $\alpha_L$ through the canonical isomorphism.

  The same definition applies for cohomological invariants.
\end{defi}

\begin{rem}
  We see that given the ambiguity of choice in hermitian Morita equivalence,
  it is not clear what extending non-even invariants would even mean. We
  could possibly define it as ``for every splitting extension, there exists
  a Morita equivalence such that...'', but it is not very satisfying (and there
  is evidence that this definition does not yield anything interesting anyway).
\end{rem}

In \cite{G1}, we define a commutative $\Z/2\Z$-graded ring
\[ \tld{GW}^\eps(A,\sigma) = GW(K) \oplus GW^\eps(A,\sigma) \]
and in \cite{G2} we show that it is actually naturally a (graded) pre-$\lambda$-ring,
with operations $\lambda^d: GW^\eps(A,\sigma)\to GW^\eps(A,\sigma)$ for $d$ odd,
and $\lambda^d: GW^\eps(A,\sigma)\to GW(K)$ for $d$ even. Furthermore, any
hermitian Morita equivalence between $(A,\sigma)$ and $(B,\tau)$ induces an
isomorphism of graded pre-$\lambda$-rings $\tld{GW}^\eps(A,\sigma)\to \tld{GW}^\eps(B,\tau)$,
which restricts to the identity on the neutral component $GW(K)$, and
when $A$ is split the operation $\lambda^d: GW^\eps(A,\sigma)\to GW(K)$ when $d$
is even corresponds to the usual $\lambda^d: GW(K)\to GW(K)$ through the isomorphism
$GW^\eps(A,\sigma)\simeq GW(K)$ coming from \emph{any} choice of Morita equivalence
(the choice does not matter since when $d$ is even $\lambda^d(\fdiag{a}h)=\lambda^d(h)$
holds for any hermitian form $h$).

\begin{prop}
  Any even Witt invariant $\alpha\in IW(\Quad_n/\sim)$ extends to
  $IW(\Herm_n^\eps(A,\sigma)/\sim)$.
\end{prop}

\begin{proof}
  We know that the $\lambda^d$ with $d$ even form a basis of $IW(\Quad_n/\sim)$.
  But from what we explained just above, the $\lambda$-operation
  $\lambda^d: GW^\eps(A,\sigma)\to GW(K)$ defines an invariant $IW(\Herm_n^\eps(A,\sigma)/\sim)$
  which extends $\lambda^d\in IW(\Quad_n/\sim)$. Therefore, all invariants extend.
\end{proof}

\begin{rem}
  This method not only shows that all invariants extend, it even gives
  a somewhat canonical extension: if $\alpha = \sum x_d\lambda^d$ in $IW(\Quad_n)$,
  then take $\hat{\alpha}=\sum x_d\lambda^d$ in $IW(\Herm_n^\eps(A,\sigma))$. This
  is uniquely defined (but it does depend on the choice of extending each $\lambda^d$
  by the corresponding operation on hermitian forms, which is definitely a very natural
  choice, but still a choice).

  When we extend an invariant in this manner, we use the same notation $\alpha$
  to denote this extension.
\end{rem}

Now we ask the question for cohomological invariants: can every even cohomological
invariants of quadratic forms be extended to $IC_0(\Herm_n^\eps(A,\sigma))$? We show
that the answer is positive when $A$ has index $2$, and that they can even be extended
to liftable invariants.

\subsection{Some Pfister forms related to quaternions}\label{sec_basic_pfis}

Let $Q$ be a quaternion algebra over $K$, with canonical involution $\gamma$.
We write $Q_0$ for the subspace of pure quaternions, and $Q_0^\times$ for the set
of invertible pure quaternions. We also use $n_Q$ to denote the norm form of $Q$,
that is the reduced norm map $Q\to K$ viewed as a quadratic form.

Let $X$ be a finite set, and let
$(h_i)_{i\in X}\in (\Herm_2^{-1}(Q,\gamma))^X$ be a family of anti-hermitian
form over $(Q,\gamma)$ of reduced dimension $2$. We want to extend to such
a family the special invariants $\Psi_0^{J,A}$ that we defined in Definition
\ref{def_psi_quad}.

Remember that they depended on morphisms $\delta: \mathcal{P}(X)\to K^\times/(K^\times)^2$
and $t: \mathcal{P}(X)\to K^\times/G(\pfis{\delta|\mathcal{P}(X)})$, which extended
the basic relation $q_i=\fdiag{t_i}\pfis{\delta_i}$. This relation does not make sense
when replacing $q_i$ with a hermitian form $h_i$, but we can still
define $\delta_i$, as it is the discriminant of the form, which is also well-defined
for hermitian forms. Precisely, if $h_i=\fdiag{z_i}_\gamma$ for some invertible pure
quaternion $z_i\in Q_0^\times$, then we may take $\delta_i=z_i^2$ (its square class does not depend on the
choice of $z_i$). So our logical naive extension of $\pfis{\delta|\mathcal{P}(I)}$
is:

\begin{defi}
  For any subset $I\subset X$, we define
  \[ \pi(I) = \pfis{z_i^2}_{i\in I}. \]
\end{defi}

Now since we only want even invariants, we only have to define our equivalent of $t$
on $\mathcal{P}_0(X)$, wich means that we should not try to generalize $t_i$ but rather
$t(\{i,j\})$, which is characterized by $q_iq_j=\fdiag{t(\{i,j\})}\pfis{\delta_i,\delta_j}$.

Now the description of $\tld{GW}(Q,\gamma)$ in \cite{G1} tells us that
\[ \fdiag{z_i}_\gamma\cdot \fdiag{z_j}_\gamma = \fdiag{-\Trd_Q(z_iz_j)}\phi_{z_i,z_j} \]
where $\phi_{z_i,z_j}$ is the unique $2$-fold Pfister form whose class in $h^2(K)$
is $(z_i^2,z_j^2)+[Q]$. This tells us that our first guess $\pi(I)$ was not the
correct one, and also what our $t(\{i,j\})$ should be.

\begin{lem}\label{lem_z_nq}
  Let $z\in Q_0^\times$. Then $\pfis{z^2}n_Q=2n_Q$ in $W(K)$.
\end{lem}

\begin{proof}
  Since $-z^2$ is the reduced norm of $z$, it is represented by $n_Q$, and therefore
  $\pfis{-z^2}n_Q=0$ in $W(K)$, which means that $\fdiag{z^2}n_Q=-n_Q$ and thus
  $\pfis{z^2}n_Q=2n_Q$ in $W(K)$.
\end{proof}

\begin{prop}
  If $|I|\pgq 2$, then $\pi(I) - 2^{|I|-2}n_Q$ is Witt-equivalent to a
  (unique) $|I|$-fold Pfister form.
\end{prop}

\begin{proof}
  Assume for simplicity that $I=\{1,\dots,n\}$. Then from the preceding lemma,
  \[ 2^{n-2}n_Q = \pfis{z_3^2,\dots,z_n^2}n_Q \]
  in $W(K)$, which means that
  \[ \pi(I) - 2^{|I|-2}n_Q = \pfis{z_3^2,\dots,z_n^2}(\pfis{z_1^2,z_2^2}-n_Q). \]

  Now if we take $z_0\in Q_0^\times$ which anti-commutes with $z_1$:
  \begin{align*}
    \pfis{z_1^2,z_2^2}-n_Q &= \pfis{z_1^2,z_2^2} - \pfis{z_1^2,z_0^2} \\
                           &= \pfis{z_1^2,z_0^2z_2^2} - \pfis{z_1^2,z_0^2,-z_2^2}
  \end{align*}
  in $W(K)$, and $\pfis{z_1^2,z_0^2,-z_2^2}$ is hyperbolic since $-z_2^2$ is
  represented by $n_Q=\pfis{z_1^2,z_0^2}$. So in fact $\pfis{z_1^2,z_2^2}-n_Q$
  is Witt-equivalent to a $2$-fold Pfister form. We conclude with the fact that two
  $n$-fold Pfister forms which are Witt-equivalent are actually isometric.
\end{proof}

This allows the following definition:

\begin{defi}
  Let $I\subset X$. If $|I|\ppq 1$ we set $\phi(I)=\pi(I)$. If $|I|\pgq 2$,
  $\phi(I)$ is the unique $|I|$-fold Pfister form which is Witt-equivalent
  to $\pi(I) - 2^{|I|-2}n_Q$.
\end{defi}

\begin{rem}
  When $Q$ is split, then $n_Q$ is hyperbolic and therefore $\phi(I)=\pi(I)$.
\end{rem}

\begin{ex}
  When $Q$ is split, and each $h_i$ corresponds to some $q_i$ though some
  Morita equivalence, then $\phi(I)=\Psi^{\emptyset,I}((q_i)_{i\in X})$ (the choice
  of Morita equivalence does not matter since this is an even invariant). So
  we can already extend $\Psi^{\emptyset,I}$ to $IW((\Herm_2^{-1}(Q,\gamma))^X)$,
  to an invariant that also takes values in Pfister forms.

  Of course $\pi(I)$ would also define such an extension (but it is
  not the relevant one for our purposes).
\end{ex}

We can record some basic facts about those forms:

\begin{lem}\label{lem_prod_pi_phi}
  Let $I,J\subset X$. Then
  \begin{align*}
    \pi(I)\cdot \pi(J) &= 2^{|I\cap J|}\pi(I\cup J) \\
    \pi(I)\cdot \phi(J) &= \left\{ \begin{array}{lc} 2^{|I\cap J|}\pi(I\cup J) & \text{if $|J|\ppq 1$} \\
                                     2^{|I\cap J|}\phi(I\cup J) & \text{otherwise}
                                   \end{array} \right. \\
    \phi(I)\cdot \phi(J) &= \left\{ \begin{array}{lc} 2^{|I\cap J|}\pi(I\cup J) & \text{if $|I|,|J|\ppq 1$} \\
                                     2^{|I\cap J|}\phi(I\cup J) & \text{otherwise.}
                                   \end{array} \right. \\
  \end{align*}
\end{lem}

\begin{proof}
  The first formula is clear once we notice that $\pfis{a,a}=\pfis{-1,a}=2\pfis{a}$.
  The other two are also very simple using Lemma \ref{lem_z_nq}.  
\end{proof}

It is also useful to know more about the case where $|I|=2$.

\begin{prop}\label{prop_phi_trace}
  Suppose $I=\{i,j\}$. Then the symmetric bilinear form
  \[ \anonfoncdef{Q\times Q}{K}{(z,z')}{\Trd_Q(zz_i\overline{z'}z_j)} \]
  is isometric to $\fdiag{\Trd_Q(z_iz_j)}\phi(I)$.
\end{prop}

\begin{proof}
  This is the characterization of $h_j\cdot h_j$ given in \cite{G1}.
\end{proof}

\subsection{The invariant $\Psi^{J,A}$}

We now generalize our morphism $t$:

\begin{prop}\label{prop_exist_t}
  We keep the notations of Section \ref{sec_basic_pfis}. There is a unique
  group morphism
  \[ t: \mathcal{P}_0(X) \to K^\times/G(\phi(X)) \]
  such that for all $i\neq j\in X$ we have $t(\{i,j\})=-\Trd_Q(z_iz_j)$.
\end{prop}

\begin{proof}
  The $\{i,j\}$ are generators of $\mathcal{P}_0(X)$ as a group (or
  equivalently as an $\mathbb{F}_2$-vector space), so there can only
  be one morphism satisfying those conditions.

  The only non-trivial thing to check is that if $i,j,k\in X$ are distinct
  then $t(\{i,j\})t(\{i,k\}) = t(\{j,k\})$,
  in other words that $T = -\Trd_Q(z_iz_j)\Trd_Q(z_iz_k)\Trd_Q(z_jz_k)$ is represented
  by the Pfister form $\phi(X)$. Since we may as well assume that 
  $X=\{i,j,k\}$, we may prove that $T$ is represented by $\phi(\{i,j,k\})$.

  Using Lemma \ref{lem_magique}, we see that
  \[ T = -(\Trd_Q(z_jz_k))^2\cdot z_i^2 + (-\Trd_Q(z_jz_k))\cdot \Trd_Q(z_iz_jz_iz_k). \]

  Now it follows from Proposition \ref{prop_phi_trace} (using $z_1=z_j$, $z_2=z_k$ and
  $z=z'=z_i$) that $-\Trd_Q(z_jz_k)\cdot \Trd_Q(z_iz_jz_iz_k)$ is represented by $\phi(\{j,k\})$.
  It is also clear that $-(\Trd_Q(z_jz_k))^2\cdot z_i^2$ is represented by $\fdiag{-z_i^2}$,
  and thus by $\fdiag{-z_i^2}\phi(\{j,k\})$.

  All in all, $T$ is represented by $\fdiag{-z_i^2}\phi(\{j,k\})\perp \phi(\{j,k\})$,
  which is $\pfis{z_i^2}\phi(\{j,k\})$, and Lemma \ref{lem_prod_pi_phi} shows that
  this form is $\phi(\{i,j,k\})$, so we are done.
\end{proof}

The following lemma is used in the proof of Proposition \ref{prop_exist_t}, but can also be
used in other contexts, so we record it here. It is a simple but powerful relation
satisfied by quaternions.

\begin{lem}\label{lem_magique}
  Let $z_1,z_2,z_3\in Q_0^\times$. Then
  \[ \Trd_Q(z_1z_2)\Trd_Q(z_1z_3) = z_1^2\Trd_Q(z_2z_3) + \Trd_Q(z_1z_2z_1z_3). \]
\end{lem}

\begin{proof}
  The basic idea is that if $z,z'\in Q_0^\times$, then $\Trd_Q(zz')=zz'+z'z$.
  Then:
  \begin{align*}
    \Trd_Q(z_1z_2)\Trd_Q(z_1z_3) &= \Trd_Q(z_1z_2)(z_1z_3+z_3z_1) \\
                                 &= \Trd_Q(z_1z_2)z_1z_3 + z_3z_1\Trd_Q(z_1z_2) \\
                                 &= (z_1z_2+z_2z_1)z_1z_3 + z_3z_1(z_1z_2+z_2z_1) \\
                                 &= z_1^2(z_2z_3+z_3z_2)+ (z_1z_2z_1z_3+z_3z_1z_1z_3)\\
                                 &= z_1^2\Trd_Q(z_2z_3) + \Trd_Q(z_1z_2z_1z_3). \qedhere
  \end{align*}
\end{proof}

\begin{rem}
  We get an analog of equation (\ref{eq_prod_q}): for any $I\in \mathcal{P}_0(X)$,
  \begin{equation}
    \prod_{i\in I}h_i = \fdiag{t(I)}\phi(I)
  \end{equation}
  where the product is taken in $\tld{GW}(Q,\gamma)$.
\end{rem}

\begin{ex}
  When $Q$ is split and $z_i$ corresponds through some Morita equivalence
  to $q_i=\fdiag{a_i,b_i}$, then $-\Trd_Q(z_iz_j)=a_ib_j+b_ia_j$. In the split
  case in Section \ref{sec_split} we had naturally taken as repesentative for $t(\{i,j\})$ 
  either $a_ia_j$, $a_ib_j$, $b_ia_j$ or $b_ib_j$, but this is another possibility, more 
  symmetrical: indeed, one can check by hand that $a_ia_j(a_ib_j+b_ia_j)$ is represented by
  $\pfis{-a_ib_i,-a_jb_j}$, so $a_ib_j+b_ia_j$ and $a_ia_j$ define the same class
  modulo $G(\phi(\{i,j\}))$.
\end{ex}

\begin{defi}\label{def_psi_herm}
  For any disjoint subsets $A,J\subset X$, we define 
  \[ \psi(J;A) = \pfis{t|W_{J,A}}\phi(A\cup J) \]
  where $W_{J,A}=(J+\mathcal{P}(A))\cap \mathcal{P}_0(X)$.
  This defines an invariant $\Psi_0^{J,A}\in IW_0((\Herm_2^{-1}(Q,\gamma))^X)$ by
  $\Psi_0^{J,A}((h_i)_{i\in X}) = \psi(J;A)$.
\end{defi}

The form $\psi(J;A)$ is well-defined since $W_{J,A}$ is an affine subset of 
$\mathcal{P}_0(A\cup J)$. It is clearly an invariant as everything is defined canonically. 
From everything we said this far we deduce:

\begin{prop}
  The invariant $\Psi_0^{J,A}\in IW_0((\Herm_2^{-1}(Q,\gamma))^X)$ given in Definition
  \ref{def_psi_herm} is an extension of $\Psi_0^{J,A}\in IW_0((\Quad_2)^X)$ defined in Definition
  \ref{def_psi_quad}, which also takes values in general $(2|A|+|J|-1)$-fold Pfister forms.
\end{prop}

\subsection{Extending cohomological invariants to quaternionic forms}\label{sec_stief_quat}

As we mentioned before equation (\ref{eq_sum_prd}), the $P_n^d$ are defined
in any pre-$\lambda$-ring, and always satisfy this equation. In particular
it holds in $\tld{GW}(Q,\gamma)$, and using the $\Z/2\Z$-grading we see
that for any anti-hermitian form $h\in \Herm_n^{-1}(Q,\gamma)$, the component
of $P_n^d(h)$ in $GW(K)$ is exactly $Q_n^d(h)$. Although only $Q_n^d$ defines
a Witt invariant, formula (\ref{eq_sum_prd}) makes $P_n^d$ a very convenient
intermediary.

By definition of $P_n^d$ and the $\lambda$-operations in $\tld{GW}(Q,\gamma)$,
we have:
\begin{align*}
  P_2^0(\fdiag{z_i}_\gamma) &= 1, \\
  P_2^1(\fdiag{z_i}_\gamma) &= 2-\fdiag{z_i}_\gamma, \\
  P_2^2(\fdiag{z_i}_\gamma) &= \pfis{z_i^2} - \fdiag{z_i}_\gamma. \\                              
\end{align*}

For any $\gamma: X\to \N$ we can also define $P_2^\gamma$ on families in 
$\Herm_{2|X|}^{-1}(Q,\gamma)$, and taking the even component defines $Q_2^\gamma$,
which is the natural extension of $Q_2^\gamma\in IW_0((\Quad_2)^X)$ to 
$IW_0((\Herm_2^{-1}(Q,\gamma))^X)$.

We first establish that the extended Witt invariant $Q_2^\gamma$ has almost the same 
combinatorial description as in Proposition \ref{prop_pq_m2} (with almost the same proof).

\begin{prop}\label{prop_q_gamma_quat}
  Let $d>2$ be even, with $d\ppq 2|X|$, and let $\gamma: X\to \{0,1,2\}$
  such that $|\gamma|=d$. We write $A=\gamma^{-1}(\{2\})$ and $B=\gamma^{-1}(\{1\})$.
  Then in $IW((\Herm_2^{-1}(Q,\gamma))^X)$ we have
  \[ Q_2^\gamma =  2^{|A\cup B|-2}n_Q + \sum_{J\subset B}2^{|B\setminus J|}\psi_0(J;A). \] 
\end{prop}

\begin{proof}
  Let us write $h=\sum_{i\in X}\fdiag{z_i}_\gamma$.
  So $Q_2^\gamma(h)$ is the component in $GW(K)$ of
  \[ P_2^\gamma((\fdiag{z_i}_\gamma)_{i\in X}) = \prod_{i\in A}(\pfis{z_i^2} - \fdiag{z_i}_\gamma)\cdot
    \prod_{i\in B}(2-\fdiag{z_i}_\gamma) \]
  When we develop this product, we
  must choose for each $i\in A\cup B$ whether we choose the term in $GW(K)$ or the term
  in $GW^{-1}(Q,\gamma)$. We can describe our choice by the subsets $I\subset A$ and $J\subset B$
  of the indices for which we chose the odd component. We must have $|I|+|J|$ even,
  so that the resulting product lands in $GW(K)$. This gives:
  \begin{align*}
    Q_2^\gamma((\fdiag{z_i}_\gamma)_{i\in X})&= \sum_{I\cup J\in \mathcal{P}_0(A\cup B)}
             2^{|B\setminus J|}\pi(A\setminus I) \fdiag{t(I\cup J)}\phi(I\cup J) \\
      &= 2^{|B|}\pi(A) + \sum_{I\cup J\in \mathcal{P}_0(A\cup B)\setminus \{\emptyset \}}
             2^{|B\setminus J|}\fdiag{t(I\cup J)}\phi(A\cup J) \\  
      &= 2^{|A\cup B|-2}n_Q + \sum_{I\cup J\in \mathcal{P}_0(A\cup B)}
             2^{|B\setminus J|}\fdiag{t(I\cup J)}\phi(A\cup J) \\  
      &= 2^{|A\cup B|-2}n_Q + \sum_{J\subset B}2^{|B\setminus J|}\left(\sum_{I\subset A,|I\cup J| \text{ even}}
                                      \fdiag{t(I\cup J)}\right)\phi(A\cup J) \\
      &= 2^{|A\cup B|-2}n_Q + \sum_{J\subset B}2^{|B\setminus J|}\psi_0(J;A). \qedhere
  \end{align*}
\end{proof}

The crucial difference between Propositions \ref{prop_pq_m2} and \ref{prop_q_gamma_quat}
is the constant term $2^{|A\cup B|-2}n_Q$. In particular, we see that $Q_2^\gamma$
does not take value in $I^d$ but only in $I^{d/2}$, but the obstruction is a known constant
that can be taken care of.

\begin{coro}
  Let $d>2$ be e.ven, with $d\ppq 2|X|$. The image of the extended invariant 
  $Q_{2|X|}^d\in IW(\Herm_{2|X|}^{-1}(Q,\gamma))$ in $IW((\Herm_2^{-1}(Q,\gamma))^X)$ is given by
  \[ \sum_{|\gamma|=d}Q_2^\gamma =  N(|X|,d)n_Q 
       + \sum_{J,A\subset X}\binom{|X|+|A|-k(J,A)}{|X|+|A|-d}2^{d-k(J,A)}\Psi^{J,A}  \]
  where $N(|X|,d)\in \N$ is divisible by $2^{d/2-2}$ and $k(J,A)=2|A|-|J|$,
  and the sum is over all $J,A\subset X$ disjoint with $2|A|\ppq d\ppq |X|+|A|$ and
  $k(J,A)\ppq d$.
\end{coro}

\begin{proof}
  The fact that $P_{2|X|}^d(\sum_{i\in X} h_i)=\sum_{|\gamma|=d}P_2^\gamma((h_i)_{i\in X})$
  is proved just as in \ref{}, and the corresponding formula for $Q_{2|X|}^d$ follows
  by taking the even part.

  Then, using Proposition \ref{prop_q_gamma_quat}, we just have to count how many times 
  $n_Q$ and each $\psi(J;A)$ appears when
  we sum over all $\gamma:X\to \{0,1,2\}$ with $|\gamma|=d$. Note that $|\gamma|=2|A|+|B|$.
  Of course the data of $\gamma$ is equivalent to the data of $A$ and $B$. So as for $n_Q$,
  we see that it appears
  \[ \sum_{A\subset X}\binom{|X|-|A|}{d-2|A|}2^{d-|A|-2} \]
  times, where the sum is over subsets $A$ such that $2|A|\ppq d\ppq |X|+|A|$.
  We call that number $N(|X|,d)$, and we see that all $2$-powers appearing in the
  sum are with exponent $d-|A|-2\pgq d/2-2$.

  For each pair $(J,A)$ of disjoint subsets of $X$, $\psi^{J,A}$ appears with factor
  $2^{|B\setminus J|}=2^{d-k(J,A)}$ for each overset $B$ of $J$ which is disjoint with
  $A$ and with $2|A|+|B|=d$. For this to exist we must have $2|A|\ppq d\ppq |X|+|A|$
  and $k(J,A)\ppq d$, and in that case there are $\binom{|X|+|A|-k(J,A)}{|X|+|A|-d}$
  possibilities for $B$ (by choosing $B\setminus J$ in $X\setminus (A\cup J)$ with
  cardinal $d-k(J,A)$).
\end{proof}

Again, we have a constant obstruction, namely $N(|X|,d)n_Q$ (which of course disappears
over any splitting field as $n_Q$ becomes hyperbolic), to  $Q_{2|X|}^d$ taking values
in $I^d$. We finally get our extension result for cohomological invariants:

\begin{thm}\label{thm_inv}
  Every even invariant $\alpha\in IC_0(\Quad_{2n})$ extends to a liftable invariant
  $\hat{\alpha}\in IC_0(\Herm_{2n}^{-1}(Q,\gamma))$.
\end{thm}

\begin{proof}
  We may assume that $\alpha\in IC_0^m(\Quad_{2n})$ for some $m\in \N$.
  According to Proposition \ref{prop_iw_sim}, $\alpha$ is liftable to
  some $\beta\in IW_0^{\pgq m}(\Quad_{2n})$.
  Then write $\beta=\sum_d a_d Q_n^d$ with $a_d\in I^{m-d+1}(K)$. We can
  extend naturally $\beta$ to hermitian forms, and then set
  \[ \hat{\beta} = \beta - \sum_d a_d N(n,d)n_Q \]
  which according to Proposition \ref{prop_qnd_quat} takes values in $I^m$ (indeed,
  it is a combination of invariants $a_d 2^{d-k(J,A)}\Psi_0^{J,A}$ and $\Psi_0^{J,A}$
    takes values in general $(k(J,A)-1)$-fold Pfister forms). Since
  $n_Q=0\in W(K)$ when $Q$ is split, the class of $\hat{\beta}$ in
  $IC_0(\Herm_{2n}^{-1}(Q,\gamma))$ extends $\alpha$.
\end{proof}

\bibliographystyle{plain}
\bibliography{stiefel_whitney_quaternions}

\end{document}